\documentclass[reqno,11pt]{amsart}

\usepackage[all]{xy}
\usepackage{tikz-cd}
\usepackage{amssymb}
\usepackage{amsgen}
\usepackage{amsmath}
\usepackage{amsthm}
\usepackage{mathrsfs}
\usepackage{cite}
\usepackage{amsfonts}

\newcommand{\skel}[1]{^{(#1)}}

\newcommand{\dom}{\mathop{\boldsymbol d}}
\newcommand{\ran}{\mathop{\boldsymbol r}}

\newcommand{\supp}{\mathop{\mathrm{supp}}}

\newcommand{\inv}{^{-1}}
\newcommand{\p}{\varphi}

\newcommand{\til}[1]{\ensuremath{\widetilde {#1}}}

\newcommand{\wh}{\widehat}

\newcommand{\module}[1]{#1\text{-}\mathbf{mod}}

\usepackage{xcolor}

\newtheorem{Thm}{Theorem}[section]
\newtheorem{Prop}[Thm]{Proposition}

\newtheorem{Lemma}[Thm]{Lemma}
{\theoremstyle{definition}
\newtheorem{Def}[Thm]{Definition}}
{\theoremstyle{remark}
\newtheorem{Rmk}[Thm]{Remark}}
\newtheorem{Cor}[Thm]{Corollary}
{\theoremstyle{remark}
}
{\theoremstyle{remark}
\newtheorem{Example}[Thm]{Example}}

{\theoremstyle{remark}
}
{\theoremstyle{remark}
}
{\theoremstyle{remark}
}

\numberwithin{equation}{section}

\title[\'Etale groupoid algebras with coefficients in a sheaf]{\'Etale groupoid algebras with coefficients in a sheaf and skew inverse semigroup rings}

\author{Daniel Gon\c{c}alves\and Benjamin Steinberg}

\address[D.~Gon\c{c}alves]{%
	Departmento de Matem\'atica\\
	Universidade Federal de Santa Catarina\\
	Florian\'{o}polis, SC, 88040-900\\
	Brazil}
	\email{daemig@gmail.com}	

\address[B.~Steinberg]{%
    Department of Mathematics\\
    City College of New York\\
    Convent Avenue at 138th Street\\
    New York, New York 10031\\
    USA}
\email{bsteinberg@ccny.cuny.edu}

\thanks{The second author thanks the Fulbright commission for its support in visiting the Federal University of Santa Catarina in Brazil. The first author was partially supported by CNPq and Capes-PrInt, Brazil.}
\date{\today}

\keywords{\'etale groupoids, inverse semigroups, rings}

\begin{document}

\begin{abstract}
Given an action $\p$ of of inverse semigroup $S$ on a ring $A$ (with domain of $\p(s)$ denoted by $D_{s^*}$) we show that if the ideals $D_e$, with $e$ an idempotent, are unital, then the skew inverse semigroup ring $A\rtimes S$ can be realized as the convolution algebra of an ample groupoid with coefficients in a sheaf of rings. Conversely, we show that the convolution algebra of an ample groupoid with coefficients in a sheaf of rings is isomorphic to a skew inverse semigroup ring of this sort. We recover known results in the literature for Steinberg algebras over a field as special cases.
\end{abstract}

\maketitle

\section{Introduction}

Convolution algebras associated to groupoids and skew rings associated to actions are driving forces in ring theory. For example, the convolution algebra associated to an ample groupoid, also known as the Steinberg algebra~\cite{Steinbergalgebra}, has been used in the study of combinatorial algebras such as Leavitt path algebras~\cite{LeavittBook}, Kumijian-Pask algebras~\cite{cp}, separated graph algebras~\cite{separatedgraphs}, among others, see for example~\cite{aradia, chuef}. Skew inverse semigroup rings are also useful tools in the study of algebras arising from combinatorial objects, see for example~\cite{goncalvesroyer, goncalvesroyer1, OinertChain}. Furthermore, both skew inverse semigroup rings and Steinberg algebras have deep connections with topological dynamics and $C^*$-algebra theory, see for example~\cite{bcfs, BE, cardia, BC, BGOR2017, Exel}.

Although both theories are intimately connected, it is safe to state that the theory of groupoid algebras is further developed than the theory of inverse semigroup skew rings. Therefore, connections between the two theories, that allow one to pass results from one setting to the other, present an immediate advance in the theory of inverse semigroup skew rings, all the while offering a new point of view and direction of development for groupoid algebras.
In the purely algebraic setting, the first connections between the theories was established in~\cite{BG} where, motivated by Exel's results in $C^*$-algebra theory (see~\cite{Exel}), the authors realize certain partial skew group rings as Steinberg algebras and show that Steinberg algebras associated to Hausdorff ample groupoids can be seen as skew inverse semigroup rings (this latter result is implicit in the arxiv version of \cite{Steinbergalgebra}, see ~\cite[Section 4.3]{arxivgroupoid}, where it is proved in
the language of covariant representations rather than that of skew inverse
semigroup rings). These results are generalized to graded ample Hausdorff groupoids in~\cite{Hazrat}. The description of Steinberg algebras as skew inverse semigroup rings is generalized to not necessarily Hausdorff groupoids in~\cite{Demeneghi}. In~\cite{groupoidbundles} the interplay between partial skew inverse semigroup rings and Steinberg algebras is studied and it is proved that skew inverse semigroup rings of the form $C_c(X,R)\rtimes S$, where $C_c(X,R)$ is the ring of compactly supported locally constant $R$-valued mappings on a locally compact Hausdorff and zero-dimensional space $X$, can be seen as Steinberg algebras of appropriate groupoids of germs. Finally, in~\cite{BC} it is shown that the category of unitary $A_R(\mathscr G)$-modules, where $A_R(\mathscr G)$ denotes the Steinberg algebra  over $R$ associated to an etale groupoid $\mathscr G$ with locally compact totally disconnected unit space, is equivalent to the category of sheaves of $R$-modules over $\mathscr G$.  This is an algebraic analogue of Renault's disintegration theorem~\cite{renaultdisintegration} for representations of groupoid $C^*$-algebras.

Our goal is to deepen the above results, obtaining a description of skew inverse semigroup rings in terms of groupoid convolution algebras. For this we pass to sheaf theory and define a convolution algebra of an ample groupoid with coefficients in a sheaf of rings (see Section~\ref{s:sheaf.coeff}). To make the interplay between groupoid algebras and skew inverse semigroup rings complete we then show that this newly defined groupoid convolution algebra can be seen as a skew inverse semigroup ring, and vice versa. We describe more precisely our work below.

In Sections~\ref{ample groupoids} and~\ref{issr} we discuss inverse semigroups, ample groupoids and skew inverse semigroup rings, setting up notation and key definitions that will be used throughout the paper. We introduce a convolution algebra $\Gamma_c(\mathscr G,\mathcal O)$, associated to a sheaf of rings $\mathcal O$ over an ample groupoid $\mathscr G$, in Section~\ref{s:sheaf.coeff}. When the sheaf in question is a constant sheaf of commutative rings, we recover the  Steinberg algebra introduced by the second author in~\cite{Steinbergalgebra}. In Section~\ref{s:as.quotient}, we show that if $\mathscr G$ is an ample groupoid, $\mathcal O$ is a $\mathscr G$-sheaf of rings and $S\leq \mathscr G^a$ (where $\mathscr G^a$ is the inverse semigroup of compact open bisections) is an inverse subsemigroup satisfying $\mathscr G\skel 1=\bigcup S$ (G1), then $\Gamma_c(\mathscr G,\mathcal O)$ is a quotient of $\Gamma_c(\mathscr G\skel 0, \mathcal O)\rtimes S$ for an appropriate  action of $S$ on the ``diagonal'' subalgebra $\Gamma_c(\mathscr G\skel 0,\mathcal O)$ of compactly supported sections of $\mathcal O$ over $\mathscr G\skel 0$ with pointwise operations.

We develop the tools to prove that the quotient map of Section~\ref{s:as.quotient} is an isormorphism (under the mild hypothesis (G2), see Section~\ref{ample groupoids}) in Section~\ref{desintegration}. To describe our key result in this section, let $R=\Gamma_c(\mathscr G,\mathcal O)$, where $\mathscr G$ is an ample groupoid and $\mathcal O$ is a $\mathscr G$-sheaf of rings. We prove that $\module{R}$ can be identified with the category of $\mathscr G$-sheaves of $\mathcal O$-modules, where  $\module{R}$ is the category of unitary (left) $R$-modules (this generalizes the disintegration theorem of~\cite{groupoidbundles}). In Section~\ref{gpdasisr} we show that unitary $\Gamma_c(\mathscr G\skel 0,\mathcal O)\rtimes S$-modules also can be represented as modules of global sections of $\mathscr G$-sheaves of $\mathcal O$-modules, where  $S\leq \mathscr G^a$ is an inverse subsemigroup satisfying  the germ conditions (G1) and (G2). With this we show that the quotient map of Section~\ref{s:as.quotient} is an isomorphism; see Theorem~\ref{t:factor.through}.

In Section~\ref{Piercerep} we generalize the Pierce representation of a ring~\cite{Pierce} as global sections of a sheaf of rings over a Stone space in two ways: we consider rings with local units and we allow smaller generalized Boolean algebras. Finally we finish the paper in Section~\ref{isrAsGpd}, where we prove the converse of Theorem~\ref{t:factor.through} by showing that every skew inverse semigroup ring (with respect to a sufficiently nice action) is isomorphic to an ample groupoid convolution algebra with coefficients in a sheaf of rings.

\section{Inverse semigroups and ample groupoids}\label{ample groupoids}
In this section, we recall some basic notions concerning inverse semigroups and ample groupoids.

\subsection{Inverse semigroups}
An \emph{inverse semigroup} is a semigroup $S$ such that, for all $s\in S$, there is a unique element $s^*\in S$ with $ss^*s=s$ and $s^*ss^*=s^*$.  Note that $s^*s$ and $ss^*$ are idempotents.  The set $E(S)$ of idempotents of $S$ is a commutative subsemigroup and is a meet semilattice via the partial ordering $e\leq f$ if $ef=e$.  This partial order extends to a compatible partial order on $S$ by putting $s\leq t$ if $ss^*t=s$ or, equivalently, $ts^*s=s$.  One can show that $(st)^*=t^*s^*$ using that idempotents commute. See~\cite{Lawson} for a detailed introduction to inverse semigroups.

Homomorphisms between inverse semigroups automatically preserve the involution and send idempotents to idempotents.  They also preserve the natural partial order discussed above.  A mapping $\p\colon S\to T$ of inverse semigroups is a homomorphism if and only if it is order preserving,  restricts to a homomorphism on idempotents and satisfies $\p(st)=\p(s)\p(t)$ whenever $s^*s=tt^*$; see~\cite[Chapter~3, Theorem~5]{Lawson}.

The prototypical example of an inverse semigroup is the semigroup of all partial bijections of a set $X$ under composition of partial mappings.  The natural partial order in this case is just the restriction ordering: $f\leq g$ if $f$ is a restriction of $g$.  Every inverse semigroup $S$ can be faithfully represented as an inverse semigroup of partial bijections of itself.

\subsection{Ample groupoids}

An \emph{\'etale groupoid} is a topological groupoid  $\mathscr G$ such that its unit space  $\mathscr G\skel 0$ is locally compact and Hausdorff and its range map $\ran$ is a local homeomorphism (this implies that the domain map $\dom$ and the multiplication map are also local homeomorphisms). A bisection of $\mathscr G$ is a subset $B\subseteq \mathscr G$ such that the restriction of the range and source maps to $B$ are injective. An \'etale groupoid is \emph{ample} if its unit space has a basis of compact open sets or, equivalently, if the arrow space $\mathscr G\skel 1$ has a basis of compact open bisections.

We denote by $\mathscr G^a$ the set of all compact open bisection of $\mathscr G$. It is an inverse semigroup with operations given by \[BC = \{ bc \in \mathscr G \mid b\in B, c\in C, \text{ and } \dom(b) = \ran(c)\}\] and $B^*=\{b^{-1}\mid b\in B\}$.  We often write $B\inv$ instead of $B^*$ in the inverse semigroup  $\mathscr G^a$.
We remark for future use that, for every open bisection $B$, of an ample groupoid $\mathscr G$, the range and source maps are homeomorphisms from $B$ to $\dom(B)$ and $\ran(B)$, respectively.

From now on, following Bourbaki, the term ``compact'' will include the Hausdorff axiom.  However, a space can be locally compact without being Hausdorff.
If $f\colon X\to Z$ and $g\colon Y\to Z$ are maps of spaces, then their \emph{pullback} is \[X\times_{f,g} Y=\{(x,y)\mid f(x)=g(y)\}\] (with the subspace topology of the product space).

Let $X$ be a topological space. Then $I_X$ denotes the inverse monoid of all homeomorphisms between open subsets of $X$.  An action of an inverse semigroup $S$ on $X$ is a homomorphism $\rho\colon S\to I_X$.  To be consistent with notation in the literature, we put $\rho(s)=\rho_s$ and write $\rho_s\colon D_{s^*}\to D_s$.  We say the action is \emph{non-degenerate} if $X=\bigcup_{e\in E(S)} D_e$.  We shall always assume that all our actions are non-degenerate. A non-degenerate action is called \emph{Boolean} if $X$ is Hausdorff, has a basis of compact open sets and $D_e$ is compact open for all $e\in E(S)$ (i.e., each $D_s$ is compact open).  Sometimes this is also referred to as an ample system in the literature.

For example, if $\mathscr G$ is an ample groupoid, then $\mathscr G^a$, the inverse semigroup of compact open bisections,  has a Boolean action on $\mathscr G\skel 0$ defined as follows.  To each $U\in \mathscr G^a$, we associate the homeomorphism  $\rho_U = \ran \circ (\dom|_U)\inv\colon \dom(U)\to \ran(U)$.  The mapping $U\mapsto \rho_U$ is a Boolean action.

Let $S\leq \mathscr G^a$ be an inverse subsemigroup. We define two conditions on $S$ as follows.
\begin{enumerate}
\item [(G1)] $\mathscr G\skel 1=\bigcup S$.
\item [(G2)] If $\gamma\in U\cap V$ with $U,V\in S$, then there is $W\in S$ with $\gamma\in W\subseteq U, V$.
\end{enumerate}
We shall say that $S\leq \mathscr G^a$ satisfies the \emph{germ conditions} if it satisfies (G1) and (G2).  Obviously $\mathscr G^a$ satisfies the germ conditions, being a basis for the topology on $\mathscr G\skel 1$.

Given a Boolean action $\rho$ of an inverse semigroup $S$ on a space $X$, we can form an ample groupoid $\mathscr G=S\ltimes X$, called the \emph{groupoid of germs} of the action.  Details can be found in~\cite{Exel,Steinbergalgebra}, we just provide the definitions.  One has $\mathscr G\skel 0= X$ and \[\mathscr G\skel 1 = \{(s,x)\in S\times X\mid x\in D_{s^*}\}/{\sim}\] where $(s,x)\sim (t,y)$ if and only if $x=y$ and there exists $u\leq s,t$ with $x\in D_{u^*}$.  We write $[s,x]$ for the class of $(s,x)$.  The groupoid structure is defined as follows.  We put $\dom([s,x]) = x$ and $\ran([s,x]) = \rho_s(x)$ (which is independent of the choice of $s$).  The product is defined by $[s,\rho_t(x)][t,x] = [st,x]$ and the inverse is given by $[s,x]\inv = [s^*,\rho_s(x)]$.  The topology on $\mathscr G\skel 0$ is that of $X$, whereas a basis of neighborhoods for $\mathscr G\skel 1$ is given by the sets $(s,U)$, where $U\subseteq D_{s^*}$ is compact open, and \[(s,U)=\{[s,x]\mid x\in U\}.\]  Note that the above sets $(s,U)$ are compact open bisections and the compact open bisections of the form $U(s)=(s,D_{s^*})$ satisfy $U(s)U(t)=U(st)$ and hence form an inverse semigroup $\til S\leq \mathscr G^a$, which is a homomorphic image of $S$.  Note that $\til S$ satisfies the germ conditions.  Indeed, if $x\in D_{s^*}$, then $[s,x]\in (s,D_{s^*})$ and so (G1) holds.  If $\gamma\in (s,D_{s^*})\cap (t,D_{t^*})$, then $\gamma=[s,x]=[t,x]$ for some $x\in D_{s^*}\cap D_{t^*}$.  Then, by definition of the germ equivalence relation $\sim$, we have $u\in S$ with $u\leq s,t$ and $x\in D_{u^*}$.  Thus $\gamma=[u,x]\in (u,D_{u^*})\subseteq (s,D_{s^*})\cap (t,D_{t^*})$, yielding (G2).

Conversely, if $\mathscr G$ is an ample groupoid and $S\leq \mathscr G^a$ satisfies the germ conditions, then Exel showed that $\mathscr G\cong S\ltimes \mathscr G\skel 0$ with respect to the natural Boolean action defined above; see~\cite{Exel}.  This explains the name ``germ conditions.''

\section{Skew inverse semigroup rings}\label{issr}
By a \emph{partial automorphism} of a ring $A$, we mean a ring isomorphism $\varphi\colon I\to J$ between two-sided ideals $I,J$ of $A$.  The collection of all partial automorphisms of $A$ forms an inverse monoid that we denote $I_A$.  If $S$ is an inverse semigroup, then an \emph{action} of $S$ on $A$ is a homomorphism $\alpha\colon S\to I_A$, usually written $s\mapsto \alpha_s$.  The domain of $\alpha(s)$ is denoted $D_{s^*}$ and the range is then $D_s$.  We say that the action is \emph{non-degenerate} if \[\sum_{e\in E(S)} D_e=A,\] a condition that we shall assume from here on out. Notice that if $S$ has an identity, then non-degeneracy is just the assumption that the identity acts as the identity morphism.

 To ensure associativity of the skew inverse semigroup ring, we assume that each $D_s$ is a ring with local units (although weaker conditions suffice). Recall that a ring $R$ has \emph{local units} if $R=\bigcup_{e\in E} eRe$, where $E$ is a set of idempotents, and the union is directed.  We call $E$ a \emph{set of local units} for $R$.  Some authors require the set $E$ to commute, but that is not necessary although it will usually be the case in this paper. Note that $R=\varinjlim_{e\in E}eRe$ in the category of rings, but the inclusions are not unit preserving.  We shall say that $R$ has central local units if it has a set of local units consisting of central idempotents.  Of course every unital ring has central local units.

Given an action $\alpha$ of $S$ on a ring $A$,
the construction of the corresponding skew inverse semigroup ring is done in three steps.
\begin{enumerate}
	\item First we consider the set
\begin{equation}\label{eq:define.L}
	\mathcal{L} = \left\{ \sum_{s\in S}^{\text{finite}} a_s \delta_s \mid \ a_s \in D_s \right\}\cong \bigoplus_{s\in S} D_s
\end{equation}
where $\delta_s$, for $s\in S$, is a formal symbol (and $0\delta_s=0$).
We equip $\mathcal{L}$ with component-wise addition and with multiplication
defined as the linear extension of the rule
\[
	(a_s \delta_s)(b_t \delta_t) = \alpha_{s}(\alpha_{s^*}(a_s) b_t) \delta_{st}.
\]
(The reader will easily verify that $\alpha_{s}(\alpha_{s^*}(a_s) b_t) \in D_{st}$.)
	\item Then, we consider the ideal
	\begin{equation}\label{eq:define.N}
	\mathcal{N} = \langle a \delta_r - a \delta_s \mid r,s \in S, \ r \leq s \text{ and } a \in D_r \rangle,
\end{equation}
i.e., $\mathcal{N}$ is the ideal of $\mathcal{L}$ generated by all elements of the form $a\delta_r - a\delta_s$, where $r\leq s$ and $a \in D_r$.   It is shown in~\cite[Lemma~2.3]{BGOR2017}, that these elements already generate $\mathcal{N}$ as an additive group.
	\item Finally, we define the corresponding
	\emph{skew inverse semigroup ring}, which we denote by $A\rtimes S$, as the quotient ring $\mathcal{L}/\mathcal{N}$.
\end{enumerate}
If $S$ is a group, then the ideal $\mathcal N$ is the zero ideal and the multiplication simplifies to the rule $a\delta_s\cdot b\delta_t = a\alpha_s(b)\delta_{st}$, and so $A\rtimes S$ is the familiar skew group ring.

Recall that an ideal $I$ of a ring $A$ is a unital ring, in its own right, if and only if $I=Ae$ with $e$ a central idempotent of $A$.  Indeed, if $e$ is a central idempotent, then trivially $Ae=eAe=AeA=eA$ is a two-sided ideal with identity $e$.  Conversely, if $I$ is a two-sided ideal with identity $e$, then for all $a\in A$, we have that $ae,ea\in I$ and so $ae=e(ae)=(ea)e=ea$.  Thus $e$ is a central idempotent.  Trivially, we then have $I=Ae$.  Later on we shall use, without comment, that any central idempotent of $I$ is also a central idempotent of $A$.  Indeed, if $f\in I$ is a central idempotent and $a\in A$, then using that $f=fe=ef$ and $ea\in I$, we have that $fa=(fe)a=f(ea)=(ea)f=(ae)f = a(ef)=af$.

A key class of inverse semigroups actions is that of spectral actions, as defined below.
\begin{Def}[Spectral action] We call an action $\alpha$ of an inverse semigroup $S$ on $A$ \emph{spectral} if it is non-degenerate and $D_e$ has a unit element $1_e$ for each $e\in E(S)$.
\end{Def}
The term ``spectral'' is used because it turns out that such actions give a Boolean action of $S$ on the Pierce spectrum of $A$.   Note that $D_{ef}=D_e\cap D_f$ implies that $1_{ef}=1_e1_f$ and so $e\mapsto 1_e$ is a homomorphism from $E(S)$ into the central idempotents of $A$.  Also, if $e\leq s^*s$, and so $D_e\subseteq D_{s^*s} = D_{s^*}$, then $\alpha_s(D_e) = \alpha_{se}(D_e)=\alpha_{se}(D_{(se)^*}) = D_{se} = D_{ses^*}$ and so $\alpha_s(1_e) = 1_{ses^*}$.

For example, if $S$ is an inverse semigroup with a Boolean action on a generalized Stone space $X$ and if $R$ is a commutative ring with unit, then we can define an action of $S$ on the ring $A=C_c(X,R)$ of compactly supported locally constant mappings $f\colon X\to R$ with pointwise operations as follows.  If $\theta\colon S\to I_X$ is the homomorphism, say $\theta_s\colon X_{s^*}\to X_s$, we let $D_s$ be  the the set of mappings $f\in A$ supported on $X_s$. Note that $D_s=A\chi_{X_s}$ and $\chi_{X_s}$ is a central idempotent.  The action is given by
\[\alpha_s(f)(x) = \begin{cases} f(\theta_{s^*}(x)), & \text{if}\ x\in X_s\\ 0 & \text{else.}\end{cases}\]  The skew inverse semigroup ring $A\rtimes S$ turns out to be isomorphic to the Steinberg algebra of the groupoid of germs $S\ltimes X$ with coefficients in $R$~\cite{Demeneghi, BC}.

If $A$ is an algebra, then $E(A)$ will denote the idempotents of $A$.  Note that $E(Z(A))$ is a generalized Boolean algebra with $e\wedge f = ef$, $e\setminus f= e-ef$ and $e\vee f = e+f-ef= e\setminus f+f\setminus e+e\wedge f$.  Let $B$ be the generalized Boolean algebra generated by the $\{1_e\mid e\in E(S)\}$.  Then our assumption that $\alpha$ is non-degenerate implies $A=\varinjlim_{e\in B} eAe=\bigcup_{e\in B} eAe$ and, hence, has local units belonging to the center of $A$.  This follows since $D(e_1)+\cdots + D(e_n)= A(e_1\vee \cdots\vee e_n)$.

\begin{Rmk}\label{r:short.form}
In the case of a spectral action, we note that if $a\in D_s$ and $b\in D_t$, then
\begin{align*}
(a\delta_s)(b\delta_t) & = \alpha_s(\alpha_{s^*}(a)b)\delta_{st}=\alpha_s(\alpha_{s^*}(a)b1_{s^*s})\delta_{st} = \alpha_{ss^*}(a)\alpha_s(b1_{s^*s})\delta_{st} \\ & = a\alpha_s(b1_{s^*s})\delta_{st}
\end{align*}
 which may look a bit more like the familiar formula for a skew group ring.
\end{Rmk}

The theory of covariant representations for  partial inverse semigroup rings was first developed in ~\cite{Vivitese} and~\cite{LCtese}. For completeness we present the main results (and some proofs) below, already adapted to our context.   We retain the notation of \eqref{eq:define.L} and \eqref{eq:define.N}.

\begin{Prop}\label{p:embed}
Let $S$ have a spectral action $\alpha$ on the ring $A$.  Then there is an embedding $\Theta\colon A\to A\rtimes S$ and a homomorphism $\Phi\colon S\to A\rtimes S$ given by \[\Theta(a) = a_1\delta_{e_1}+\cdots+a_n\delta_{e_n}+\mathcal N\] where $a=a_1+\cdots +a_n$ with $a_i\in D_{e_i}$ and \[\Phi(s) = 1_{ss^*}\delta_s+\mathcal N.\]  Moreover, $A\rtimes S$ has a set of local units contained in $\Theta(B)$ where $B$ is the generalized Boolean algebra generated by $\{1_e\mid e\in E(S)\}$.  Furthermore,  $\Theta(\alpha_s(a)) = \Phi(s)\Theta(a)\Phi(s^*)$ for all $s\in S$ and $a\in D_{s^*}$ and, for $e\in E(S)$, we have that $\Theta(1_e) = \Phi(e)$.
\end{Prop}
\begin{proof}
It is easy to see that $\Phi$ is a homomorphism as
\begin{align*}
\Phi(s)\Phi(t) &= (1_{ss^*}\delta_s)(1_{tt^*}\delta_t)+\mathcal N = \alpha_s(\alpha_{s^*}(1_{ss^*})1_{tt^*})\delta_{st}+\mathcal N \\ & = \alpha_s(1_{s^*s}1_{tt^*})\delta_{st}+\mathcal N = \alpha_s(1_{s^*stt^*})\delta_{st}+\mathcal N \\ &= 1_{stt^*s^*}\delta_{st}+\mathcal N = \Phi(st).
\end{align*}
The fact that $\Theta$ is an embedding is proved in~\cite[Proposition~3.1]{BGOR2017}.

We compute that if $a\in D_{s^*}$, then
\begin{align*}
\Phi(s)\Theta(a)\Phi(s^*) &= (1_{ss^*}\delta_s)(a\delta_{s^*s})(1_{s^*s}\delta_{s^*})+\mathcal N \\ & = (\alpha_s(\alpha_{s^*}(1_{ss^*})a)\delta_s)(1_{s^*s}\delta_{s^*})+\mathcal N \\ &= (\alpha_s(a)\delta_s)(1_{s^*s}\delta_{s^*})+\mathcal N \\ & = \alpha_s(\alpha_{s^*s}(a)1_{s^*s})\delta_{ss^*}+\mathcal N \\ &= \alpha_s(a)\delta_{ss^*}+\mathcal N = \Theta(\alpha_s(a)).
\end{align*}

If $e\in E(S)$, then $\Theta(1_e) = 1_e\delta_e +\mathcal N= \Phi(e)$.

Since $B$ admits joins to show that $\theta(B)$ is a set of local units, it suffices to show that each element of the form $a\delta_s+\mathcal N$ belongs to $\Theta(e)(A\rtimes S)\Theta(e)$ for some $e\in B$.  First note that $\Theta(1_{ss^*})(a\delta_s+\mathcal N) = (1_{ss^*}\delta_{ss^*})(a\delta_s)+\mathcal N =a\delta_s+\mathcal N$ and $(a\delta_s+\mathcal N)\Theta(1_{s^*s})=(a\delta_s)(1_{s^*s}\delta_{s^*s})+\mathcal N  = a\delta_s+\mathcal N$.  It now follows that if $e=1_{s^*s}\vee 1_{ss^*}$, then $\Theta(e)(a\delta_s+\mathcal N)\Theta(e)=a\delta_s+\mathcal N$.
\end{proof}

The properties of the mappings $\Theta$ and $\Phi$ above are sufficiently important to give them a name.

\begin{Def}[Covariant system]
Let $S$ be an inverse semigroup with a spectral action $\alpha$ on a ring $A$.  Then a \emph{covariant system} for $(S,A,\alpha)$ consists of a triple $(R,\theta,\p)$ where $R$ is a ring and $\theta\colon A\to R$, $\p\colon S\to R$ are homomorphisms such that:
\begin{itemize}
\item [(C1)] $\theta(\alpha_s(a)) = \p(s)\theta(a)\p(s^*)$ for all $a\in D_{s^*}$;
\item [(C2)] $\theta(1_e) = \p(e)$ for $e\in E(S)$.
\end{itemize}
\end{Def}

For example, $(A\times S, \Theta,\Phi)$ is  a covariant system.  It turns out to be the universal covariant system.

\begin{Thm}\label{t:adjunction}
Let $S$ be an inverse semigroup with a spectral action $\alpha$ on a ring $A$.  Let $R$ be a ring.  Then there is a bijection between ring homomorphisms $\pi\colon A\rtimes S\to R$ and covariant systems $(R,\theta,\p)$.  More precisely, if $\pi\colon A\rtimes S\to R$ is a ring homomorphism, then $(R,\pi\circ \Theta,\pi\circ \Phi)$ is a covariant system.  Conversely, if $(R,\theta,\p)$ is a covariant system, then there is a unique homomorphism $\pi=\theta\rtimes \p\colon A\rtimes S\to R$ such that
\[\begin{tikzcd}
A\ar{rr}{\Theta}\ar{rd}[swap]{\theta} & & A\rtimes S\ar{ld}{\pi} &\text{and} & S\ar{rr}{\Phi}\ar{rd}[swap]{\p} & & A\rtimes S\ar{ld}{\pi}\\
                                & R&                        &&                          &R&
\end{tikzcd}\]
commute.
\end{Thm}
\begin{proof}
It is clear from Proposition~\ref{p:embed} that if $\pi\colon A\rtimes S\to R$ is a homomorphism, then $(R,\pi\circ \Theta,\pi\circ \Phi)$ is a covariant system.  Suppose next that $(R,\theta,\p)$ is a covariant system.  Note that if $a\in D_s$, then $a\delta_s+\mathcal N = (a\delta_{ss^*})(1_{ss^*}\delta_s)+\mathcal N=\Theta(a)\Phi(s)$.  It follows that if $\pi\colon A\rtimes S\to R$ is a homomorphism with $\pi\circ \Theta=\theta$ and $\pi\circ \Phi=\p$, then we must define $\pi(a\delta_s+\mathcal N) = \theta(a)\p(s)$ for $a\in D_s$.  We need to show that this, in fact, gives a well defined ring homomorphism.

First we show that $\rho\colon \mathcal L\to R$ given by $\rho(a\delta_s) = \theta(a)\p(s)$ for $a\in D_s$ is a ring homomorphism.  Indeed, we have that $(a\delta_s)(b\delta_t) = \alpha_s(\alpha_{s^*}(a)b)\delta_{st}$ so we need to check that $\theta(a)\p(s)\theta(b)\p(t) = \theta(\alpha_s(\alpha_{s^*}(a)b))\p(st)$.    First we observe that $\theta(a) = \theta(1_{ss^*}a) = \theta(1_{ss^*})\theta(a) = \p(ss^*)\theta(a) = \p(s)\p(s^*)\theta(a)$.  Thus we have
\begin{align*}
\theta(a)\p(s)\theta(b)\p(t) & = \p(s)\p(s^*)\theta(a)\p(s)\theta(b)\p(t) = \p(s)\theta(\alpha_{s^*}(a))\theta(b)\p(t)\\ &=\p(s)\theta(\alpha_{s^*}(a)b)\p(t).
\end{align*}
Since $\alpha_{s^*}(a)b\in D_{s^*}\cap D_t\subseteq D_{s^*} = D_{s^*s}$, we have that $\alpha_{s^*}(a)b=\alpha_{s^*}(a)b1_{s^*s}$ and so $\theta(a)\p(s)\theta(b)\p(t)$ equals
\begin{align*}
\p(s)\theta(\alpha_{s^*}(a)b)\p(t) &=\p(s)\theta(\alpha_{s^*}(a)b)\theta(1_{s^*s})\p(t) =\p(s)\theta(\alpha_{s^*}(a)b)\p(s^*s)\p(t) \\ &= \p(s)\theta(\alpha_{s^*}(a)b)\p(s^*)\p(st)
=\theta(\alpha_s(\alpha_{s^*}(a)b))\p(st)
\end{align*}
  as required.

 Now we must show that $\mathcal N\subseteq \ker \rho$.  Let $a\in D_s$ with $s\leq r$ in $S$.  Then $s=ss^*r$ and we need to show that $\rho(a\delta_s) = \rho(a\delta_r)$.  But then $\rho(a\delta_s) = \theta(a)\p(s) =\theta(a)\p(ss^*r) = \theta(a)\p(ss^*)\p(r) = \theta(a)\theta(1_{ss^*})\p(r) = \theta(a1_{ss^*})\p(r)  =\theta(a)\p(r) = \rho(a\delta_r)$ as $a\in D_s=D_{ss^*}$.  This completes the proof.
\end{proof}

\begin{Rmk} The proof of the above result in the context of partial inverse semigroup actions can be found in~\cite[Theorem~1.6.19]{Vivitese} and~\cite[Theorem~4.3.15]{LCtese}.
\end{Rmk}

\section{Ample groupoid convolution algebras with coefficients in a sheaf of rings}\label{s:sheaf.coeff}
In this section we introduce a convolution algebra associated to a sheaf of rings over an ample groupoid.  When the sheaf in question is a constant sheaf of commutative rings, we recover the so-called Steinberg algebra introduced by the second author in~\cite{Steinbergalgebra}.  We shall see later that such convolution algebras are skew inverse semigroup rings and that, conversely, a large class of skew inverse semigroup rings are of this form.  We will show  that modules for the skew inverse semigroup ring can be identified with sheaves of modules over the sheaf of rings.  This will allow the geometric approach to modules initiated in~\cite{groupoidbundles}, and further studied in~\cite{groupoidprimitive,gcrgroupoid,EffrosHahngpd}, to be applied to skew inverse semigroup rings.

Let $\mathscr G$ be an ample groupoid.  Then a \emph{$\mathscr G$-sheaf} $\mathcal E$ consists of a topological space $E$, a local homeomorphism $p\colon E\to \mathscr G\skel 0$ and a continuous map $\alpha\colon \mathscr G\skel 1\times_{\dom,p} E\to E$ (written $(\gamma,e)\mapsto \alpha_{\gamma}(e)$) satisfying the following axioms:
\begin{itemize}
\item [(S1)] $\alpha_{p(e)}(e) = e$;
\item [(S2)] $p(\alpha_{\gamma}(e))=\ran(\gamma)$ if $\dom(\gamma)=p(e)$;
\item [(S3)] $\alpha_{\beta}(\alpha_{\gamma}(e)) = \alpha_{\beta\gamma}(e)$ whenever $\dom(\beta)=\ran(\gamma)$ and $\dom(\gamma)=p(e)$.
\end{itemize}								
If $x\in \mathscr G\skel 0$, then $\mathcal E_x=p^{-1}(x)$ is called the \emph{stalk} of $\mathcal E$ at $x$.  Notice that $\alpha_{\gamma}\colon \mathcal E_{\dom(\gamma)}\to \mathcal E_{\ran(\gamma)}$ is a bijection with inverse $\alpha_{\gamma\inv}$.    The assignment $x\to \mathcal E_x$ is a functor from $\mathscr G$ to the category of sets, and so a $\mathscr G$-sheaf is the topological analogue of such a functor.

A morphism of $\mathscr G$-sheaves $\mathcal E=(E,p,\alpha)$ and $\mathcal F=(F,q,\beta)$ is a continuous mapping $h\colon E\to F$ such that
\[\begin{tikzcd}
E\arrow{rd}[swap]{p}\arrow{rr}{h}& & F\arrow{ld}{q}\\ &\mathscr G\skel 0 &
\end{tikzcd}\]
commutes and $h(\alpha_{\gamma}(e)) = \beta_{\gamma}(h(e))$ for all $e\in E$ and $\gamma\in \mathscr G\skel 1$ with $\dom(\gamma) = p(e)$. Note that $h$ is automatically a local homeomorphism.

We shall be interested in sheaves with extra structure.  A \emph{$\mathscr G$-sheaf of (unital) rings} is a $\mathscr G$-sheaf $\mathcal O=(E,p,\alpha)$ equipped with a unital ring structure on each stalk $\mathcal O_x$ such that the following axioms hold:
\begin{itemize}
\item [(SR1)] $+\colon E\times_{p,p} E\to E$ is continuous;
\item [(SR2)] $\cdot\colon E\times_{p,p} E\to E$ is continuous;
\item [(SR3)] the unit section $x\mapsto 1_x$ is a continuous mapping $\mathscr G\skel 0\to E$;
\item [(SR4)] $\alpha_{\gamma}\colon \mathcal O_{\dom(\gamma)}\to \mathcal O_{\ran(\gamma)}$ is a ring homomorphism for all $\gamma\in \mathscr G\skel 1$.
\end{itemize}

Notice that $x\mapsto \mathcal O_x$ is a functor from $\mathscr G$ to the category of unital rings, so a $\mathscr G$-sheaf is a topological analogue of such a functor. Note that each $\alpha_{\gamma}$ is a ring isomorphism. One can prove that the zero section $x\mapsto 0_x$ is continuous and that the negation map is continuous (these are standard facts about sheaves of abelian groups, and hence rings, over spaces,~cf.~\cite{Dowker}).

\begin{Example}\label{ex:constant.sheaf}
A key example is the following.  Let $R$ be any unital ring, which we view as a space with the discrete topology.  We define the \emph{constant sheaf} of rings $\Delta(R)$ to be the $\mathscr G$-sheaf of rings with $E=R\times \mathscr G\skel 0$ and with $p\colon R\times \mathscr G\skel 0\to\mathscr G\skel 0$ the projection.  The addition and multiplication are pointwise, that is, $(r,x)+(r',x) = (r+r',x)$ and $(r,x)(r',x) = (rr',x)$.  The mapping $\alpha$ is given by $\alpha(\gamma)(r,\dom(\gamma)) = (r,\ran(\gamma))$.  It will turn out that the convolution algebra we associate to $\Delta(R)$ will be the usual algebra of $\mathscr G$ over $R$, from~\cite{Steinbergalgebra}, without the restriction on $R$ being commutative.  Viewed as a functor to the category of rings, $\Delta(R)$ takes each object of $\mathscr G$ to $R$ and each arrow to the identity map on $R$.
\end{Example}

\begin{Example}
If $G$ is a discrete group, viewed as a one-object ample groupoid, then a $G$-sheaf of rings is nothing more than a unital ring $A$ together with an action of $G$ on $A$ by automorphisms.  The convolution algebra that we shall define reduces in this case to the skew group ring $A\rtimes G$.
\end{Example}

We now aim to define the \emph{ring of global sections of $\mathcal O$ with compact support}, which we shall also call the \emph{convolution algebra of $\mathscr G$ with coefficients in the sheaf of rings $\mathcal O$}.   Because we do not assume that $\mathscr G$ is Hausdorff, defining the algebra correctly involves some subtleties.  Since our groupoids are ample, we can take some short cuts.

Let $\mathcal O=(E,p,\alpha,+,\cdot)$ be a $\mathscr G$-sheaf of rings.  Let $A(\mathscr G,\mathcal O)$ be the set of all mappings $f\colon \mathscr G\skel 1\to E$ such that $p\circ f=\ran$, that is, $f(\gamma)\in \mathcal O_{\ran(\gamma)}$ for all $\gamma\in \mathscr G\skel 1$.  We can define a binary operation on $A(\mathscr G,\mathcal O)$ by putting $(f+g)(\gamma)= f(\gamma)+g(\gamma)$, which we refer to as pointwise addition.   We shall denote by $0$ the mapping $0(\gamma) = 0_{\ran(\gamma)}$ for all $\gamma\in \Gamma$.  We can identify $A(\mathscr G,\mathcal O)$ with  $\prod_{\gamma\in \mathscr G\skel 1}\mathcal O_{\ran(\gamma)}$ and with this identification, pointwise addition becomes the usual addition in a direct product.  Hence we may conclude the following.

\begin{Prop}\label{p:first.ab.group}
The set $A(\mathscr G,\mathcal O)$ is an abelian group with respect to pointwise addition with $0$ as the identity and $(-f)(\gamma) = -f(\gamma)$ for $\gamma\in \mathscr G\skel 1$.
\end{Prop}

We define, as an abelian group, $\Gamma_c(\mathscr G,\mathcal O)$ to be the subgroup generated by all mappings $f\in A(\mathscr G,\mathcal O)$ such that there is a compact open bisection $U$ with $f|_U$ continuous and $f|_{\mathscr G\skel 1\setminus U} = 0$.  In this case, we say that $f$ is \emph{supported} on $U$. If $U$ is a compact open bisection and $s\colon \ran(U)\to E$ is any (continuous) section of $p$, then we can define an element $s\chi_U\in \Gamma_c(\mathscr G,\mathcal O)$, supported on $U$, by
\[(s\chi_U)(\gamma) = \begin{cases} s(\ran(\gamma)), & \text{if}\ \gamma\in U\\ 0_{\ran(\gamma)}, & \text{else.}\end{cases} \]  In the special case that $s$ is the unit section $x\mapsto 1_x$ over $U$, we denote $s\chi_U$ by simply $\chi_U$.  In other words,
\[\chi_U(\gamma) = \begin{cases} 1_{\ran(\gamma)}, & \text{if}\ \gamma\in U\\ 0_{\ran(\gamma)}, & \text{else.}\end{cases}\]

Notice that if $f\in \Gamma_c(\mathscr G,\mathcal O)$ is supported on a compact open bisection $U$, then $f=s\chi_U$ where $s=f\circ (\ran|_U)\inv$.  Thus $\Gamma_c(\mathscr G,\mathcal O)$ can also be described as the abelian group generated by all elements of the form $s\chi_U$ where $s\colon \ran(U)\to E$ is a section, and $U$ is a compact open bisection.

The reader should verify that if $\mathscr G$ is Hausdorff, then $\Gamma_c(\mathscr G,\mathcal O)$ consists of all continuous mappings $f\colon \mathscr G\skel 1\to E$ such that $p\circ f= \ran$ and $f$ has  compact support (i.e., $\{\gamma\mid f(\gamma)\neq 0_{\ran(\gamma)}\}$ is compact).

\begin{Example}
Let $R$ be a unital ring.  Then $A(\mathscr G,\Delta(R))$ can be identified with $R^{\mathscr G\skel 1}$ by sending $f\in R^{\mathscr G\skel 1}$  to the mapping in $A(\mathscr G,\Delta(R))$ given by $F(\gamma) = (f(\gamma),\ran(\gamma))$.  Under this identification,   $\Gamma_c(\mathscr G,\Delta(R))$ corresponds to the $R$-submodule of $R^{\mathscr G\skel 1}$ spanned by the characteristic functions $\chi_U$ with $U$ a compact open bisection.  In particular, if $\mathscr G$ is Hausdorff, it can be identified with the locally constant functions $\mathscr G\skel 1\to R$ with compact support.
\end{Example}

\begin{Example}
If $G$ is a discrete group, viewed as a one-object ample groupoid, acting on a unital ring $A$ (viewed as a sheaf of rings), then the additive group $\Gamma_c(G,A)$ is the group of finitely supported functions $G\to A$ with pointwise addition.
\end{Example}

A crucial property of elements of $\Gamma_c(\mathscr G,\mathcal O)$ is that they can only be non-zero on finitely many points of any fiber of $\dom$ or $\ran$.

\begin{Prop}\label{p:finiteness.prop}
Let $f\in \Gamma_c(\mathscr G,\mathcal O)$ and $x\in \mathscr G\skel 0$.  Then there are only finitely many $\gamma\in \dom\inv (x)$ such that $f(\gamma)\neq 0$ and, similarly, for $\ran\inv(x)$.
\end{Prop}
\begin{proof}
If $U$ is a compact open bisection and $s\colon \ran(U)\to E$ is a section, then $s\chi_U$ is non-zero on at most one element of $\dom\inv(x)$ (respectively, $\ran\inv(x)$).  Since any element of $\Gamma_c(\mathscr G,\mathcal O)$ is a finite sum of such elements, the proposition follows.
\end{proof}

We now define convolution of elements of $\Gamma_c(\mathscr G,\mathcal O)$.  The reader can then perform the straightforward verification that convolution on $\Gamma_c(\mathscr G,\Delta(R))$ is the familiar convolution from~\cite{Steinbergalgebra}.
If $f,g\in \Gamma_c(\mathscr G,\mathcal O)$ and $\gamma\in \mathscr G\skel 1$, we define their \emph{convolution} by
\begin{equation}\label{eq:define.conv}
f\ast g(\gamma) = \sum_{\beta\rho=\gamma}f(\beta) \alpha_{\beta}(g(\rho)).
\end{equation}

For example, if $G$ is a discrete group (viewed as a one-object ample groupoid) acting on a unital ring $A$ (which we view as a sheaf), then an element $f\in \Gamma_c(G,A)$ can be written as a finite formal sum $\sum_{g\in G}f(g)\delta_g$.  With this notation, the convolution product \eqref{eq:define.conv} becomes
\[\left(\sum_{g\in G}a_g\delta_g\right)\left(\sum_{g\in G}b_g\delta_g\right) = \sum_{hk=g}a_h\alpha_h(b_k)\delta_g\] and so $\Gamma_c(G,A)$ is the usual skew group ring.

We must show that \eqref{eq:define.conv} makes sense.

\begin{Prop}\label{p:well.def}
If $f,g\in \Gamma_c(\mathscr G,\mathcal O)$, then $f\ast g\in \Gamma_c(\mathscr G,\mathcal O)$.
\end{Prop}
\begin{proof}
Suppose first that $\beta\rho=\gamma$.  Then note that $g(\rho)\in \mathcal O_{\ran(\rho)} = \mathcal O_{\dom(\beta)}$ and so
$\alpha_{\beta}(g(\rho))\in \mathcal O_{\ran(\beta)}$.  Thus $f(\beta)\alpha_{\beta}(g(\rho))\in \mathcal O_{\ran(\beta)}=\mathcal O_{\ran(\gamma)}$.  The sum in \eqref{eq:define.conv} is finite by Proposition~\ref{p:finiteness.prop} as we must have $\ran(\beta)=\ran (\gamma)$ and $\dom(\rho)=\dom(\gamma)$.
It remains to verify that $f\ast g\in \Gamma_c(\mathscr G,\mathcal O)$.  Using that $\alpha_{\beta}$ is an additive homomorphism, for each $\beta\in \mathscr G\skel 1$, and that $\mathcal O_{\ran(\gamma)}$ is a ring and hence satisfies the distributive law, it suffices to show that if $U,V$ are compact open bisections and $f,g\in \Gamma_c(\mathscr G,\mathcal O)$ are supported on $U,V$, respectively,   then $f\ast g\in \Gamma_c(\mathscr G,\mathcal O)$.

First note that $(f\ast g)(\gamma)=0$ unless $\gamma\in UV$ by \eqref{eq:define.conv} and that $UV$ is a compact open bisection. Thus $f\ast g$ is supported on $UV$.  Next note that  the multiplication mapping $\mu\colon U\times_{\dom,\ran} V\to UV$ is a homeomorphism in an ample groupoid, when $U,V$ are compact open bisections.  The mapping $(f\ast g)|_{UV}$ is the composition of $(\mu|_{U\times_{\dom,\ran} V})\inv$ with the continuous mapping $U\times_{\dom,\ran} V\to E$ given by $(\beta,\rho)\mapsto f|_U(\beta)\alpha(\beta,g|_V(\rho))$ and hence is continuous. This completes the proof.
\end{proof}

\begin{Rmk}\label{r:conv}
The proof of Proposition~\ref{p:well.def} shows that if $f$ is supported on $U$ and $g$ is supported on $V$, with $U,V$ compact open bisections, then $f\ast g$ is supported on $UV$, a fact that shall be used later.
\end{Rmk}

Note that if $x\in \mathscr G\skel 0$, then $\alpha_x$ is an identity map.  It follows that if $f,g$ are supported on $\mathscr G\skel 0$, then so is $f\ast g$ and, moreover, $(f\ast g)(x)=f(x)g(x)$.  In other words, we have a subring $\Gamma_c(\mathscr G\skel 0,\mathcal O)$ consisting of the compactly supported global sections of $\mathcal O$, viewed as a sheaf of rings over the space $\mathscr G\skel 0$, with pointwise operations. Notice that the idempotents $\chi_U$ with $U\subseteq \mathscr G\skel 0$ compact open are central in the subring $\Gamma_c(\mathscr G\skel 0, \mathcal O)$.   We can view $\Gamma_c(\mathscr G,\mathcal O)$ as a left $\Gamma_c(\mathscr G\skel 0,\mathcal O)$-module via left multiplication, and the characteristic functions $\chi_U$  with $U$ compact open bisections span $\Gamma_c(\mathscr G,\mathcal O)$ as a left $\Gamma_c(\mathscr G\skel 0,\mathcal O)$-module.

\begin{Prop}\label{p:inverse.embed}
Let $U,V$ be compact open bisections. Then $\chi_U\ast \chi_V = \chi_{UV}$.
\end{Prop}
\begin{proof}
Indeed, $(\chi_U\ast \chi_V)(\gamma) = \sum_{\beta\rho=\gamma}\chi_U(\beta)\alpha_{\beta}(\chi_V(\rho))$. If $\gamma\notin UV$, then this resulting summation is zero.  Otherwise, $\gamma=\beta\rho$ for a unique $\beta\in U$ and $\rho\in V$.  Then $\chi_U(\beta)\alpha_{\beta}(\chi_V(\rho))=1_{\ran(\beta)}\alpha_{\beta}(1_{\ran(\rho)}) = 1_{\ran( \beta)}=1_{\ran(\gamma)}$ and so $\chi_U\ast \chi_V=\chi_{UV}$.
\end{proof}

\begin{Prop}\label{p:is.ring}
Let $\mathscr G$ be an ample groupoid and $\mathcal O$ a $\mathscr G$-sheaf of rings. Then $\Gamma_c(\mathscr G,\mathcal O)$ is a ring and $\Gamma_c(\mathscr G\skel 0,\mathcal O)$ is a subring (the latter with pointwise operations).
\end{Prop}
\begin{proof}
We will verify the associative law for $\ast$.
Let $f,g,h\in \Gamma_c(\mathscr G,\mathcal O)$.  We then compute
\begin{align*}
(f\ast (g\ast h))(\gamma) &= \sum_{\lambda\nu\rho=\gamma} f(\lambda)\alpha_{\lambda}(g(\nu)\alpha_{\nu}(h(\rho)))\\
&= \sum_{\lambda\nu\rho=\gamma} f(\lambda)\alpha_{\lambda}(g(\nu))\alpha_{\lambda\nu}(h(\rho))\\
&= ((f\ast g)\ast h)(\gamma).
\end{align*}
The remaining verifications are left to the reader.
\end{proof}

It turns out, as was the case for Steinberg algebras over rings, that having a multiplicative identity is equivalent to compactness of the unit space.

\begin{Prop}\label{p:unital}
The ring $\Gamma_c(\mathscr G,\mathcal O)$ is unital if and only if $\mathscr G\skel 0$ is compact, in which case the identity is given by the unit section $\chi_{\mathscr G\skel 0}$.
\end{Prop}
\begin{proof}
Assume first that $\mathscr G\skel 0$ is compact.  Then $\chi_{\mathscr G\skel 0}\in \Gamma_c(\mathscr G,\mathcal O)$ and
\begin{align*}
\chi_{\mathscr G\skel  0}\ast f(\gamma) &= 1_{\ran(\gamma)}\alpha_{\ran(\gamma)}(f(\gamma))=f(\gamma)\\
f\ast \chi_{\mathscr G\skel 0}(\gamma) &= f(\gamma)\alpha_{\gamma}(1_{\dom(\gamma)}) = f(\gamma)1_{\ran(\gamma)} = f(\gamma).
\end{align*}

Conversely, suppose that $\Gamma_c(\mathscr G,\mathcal O)$ has a unit $u$.  We show that $u=\chi_{\mathscr G\skel 0}$ (i.e., sends $x$ to $1_x$ for a unit $x$ and sends  a non-unit $\gamma$ to $0_{\ran(\gamma)}$).  Indeed, let $\gamma\in \mathscr G\skel 1$ and choose $U\subseteq \mathscr G\skel 0$ compact open with $\dom(\gamma)\in U$.  Then if $\gamma\neq \dom(\gamma)$, we have that
\[0_{\ran(\gamma)}=\chi_U(\gamma) = (u\ast \chi_U)(\gamma) =u(\gamma)\alpha_{\gamma}(1_{\dom(\gamma)}) = u(\gamma). \]  On the other hand, if $\gamma\in \mathscr G\skel 0$, then
\[1_{\gamma} = \chi_U(\gamma) = (u\ast \chi_U)(\gamma) = u(\gamma)\alpha_{\gamma}(1_{\gamma}) = u(\gamma).\]
But if $\chi_{\mathscr G\skel 0}\in \Gamma_c(\mathscr G, \mathcal O)$, then
\[\chi_{\mathscr G\skel 0} = \sum_{i=1}^n s_i\chi_{U_i}\] with $U_1,\ldots, U_n$ compact open bisections and $s_i$ a section over $\ran(U_i)$.  But then $\mathscr G\skel 0\subseteq U_1\cup\cdots \cup U_n$ and so $\mathscr G\skel 0 = \dom(U_1)\cup \cdots \cup \dom(U_n)$.  As each $\dom(U_i)$ is compact, we deduce that $\mathscr G\skel 0$ is compact.
\end{proof}

As a corollary, we deduce that $\Gamma_c(\mathscr G,\mathcal O)$ has local units.

\begin{Cor}\label{c:local.units}
Let $\mathscr G$ be an ample groupoid and $\mathcal O$ a $\mathscr G$-sheaf of rings. Then $\Gamma_c(\mathscr G,\mathcal O)$ has local units contained in the center of $\Gamma_c(\mathscr G\skel 0,\mathcal O)$.
\end{Cor}
\begin{proof}
Let $U\subseteq \mathscr G\skel 0$ be compact open.  Let $\mathscr G|_U$ be the open subgroupoid with object set $U$ and all arrows between elements of $U$.  Let $\mathcal O_U$ be the restriction of $\mathcal O$ to $U$.  So if $\mathcal O=(E,p,\alpha,+,\cdot)$, then $\mathcal O_U= (p\inv(U),p,\alpha,+,\cdot)$ (where the structure maps are appropriately restricted).  It is immediate from the definitions that $\Gamma_c(\mathscr G|_U,\mathcal O_U)$ is a subring of $\Gamma_c(\mathscr G,\mathcal O)$ and that $\Gamma_c(\mathscr G,\mathcal O) = \varinjlim_{U} \Gamma_c(\mathscr G|_U,\mathcal O_U)$ since $\mathscr G\skel 0$ has a basis of compact opens.  Each of the rings  $\Gamma_c(\mathscr G|_U,\mathcal O_U)$ is unital by Proposition~\ref{p:unital} with identity  $\chi_U\in \Gamma_c(\mathscr G\skel 0,\mathcal O)$, whence $\Gamma_c(\mathscr G|_U,\mathcal O)$ has local units contained in the center of $\Gamma_c(\mathscr G\skel 0,\mathcal O)$, as required.
\end{proof}

We now aim to generalize the characterization of the center of an ample groupoid algebra from~\cite{Steinbergalgebra} to the current setting.
\begin{Def}[Class function]
Let $\mathcal O$ be a $\mathscr G$-sheaf of rings for an ample groupoid $\mathscr G$.  An element $f\in \Gamma_c(\mathscr G,\mathcal O)$ is a \emph{class function} if the following hold.
\begin{enumerate}
\item $f(\gamma)\neq 0$ implies $\dom(\gamma)=\ran(\gamma)$ and $af(\gamma) = f(\gamma)\alpha_{\gamma}(a)$ for all $a\in \mathcal O_{\dom(\gamma)}=\mathcal O_{\ran(\gamma)}$.
\item $\alpha_{\sigma}(f(\sigma\inv \gamma\sigma)) = f(\gamma)$ for all $\sigma,\tau$ with $\ran(\sigma) = \dom(\gamma)=\ran(\gamma)$.
\end{enumerate}
\end{Def}
This definition reduces to the one in~\cite{Steinbergalgebra} when $\mathcal O$ is a constant sheaf of commutative rings.

\begin{Thm}\label{t:center}
Let $\mathscr G$ be an ample groupoid and $\mathcal O$ be a $\mathscr G$-sheaf of rings.  Then $f\in Z(\Gamma_c(\mathscr G,\mathcal O))$ if and only if $f$ is a class function.
\end{Thm}
\begin{proof}
Assume first that $f$ is a class function and let $g\in \Gamma_c(\mathscr G,\mathcal O)$.   Then we have
\begin{align}\label{eq:center.1}
f\ast g(\gamma) &= \sum_{\sigma\tau=\gamma}f(\sigma)\alpha_{\sigma}(g(\tau))\\ \label{eq:center.2}
g\ast f(\gamma) &= \sum_{\lambda\nu=\gamma}g(\lambda)\alpha_{\lambda}(f(\nu)).
\end{align}
Any non-zero term in the summand in \eqref{eq:center.1} must have $\dom(\sigma)=\ran(\sigma)$ and $f(\sigma)\alpha_{\sigma}(g(\tau)) = g(\tau)f(\sigma)$ by the first condition in the definition of a class function.  By the second condition in the definition of a class function, $f(\sigma) = \alpha_{\tau}(f(\tau\inv \sigma\tau))$.  Thus if we fix $\tau$ and  perform the invertible change of variables $\nu= \tau\inv\sigma \tau$, we obtain that
\begin{equation}\label{eq:center.3}
 \sum_{\sigma\tau=\gamma}f(\sigma)\alpha_{\sigma}(g(\tau))= \sum_{\dom(\sigma)=\ran(\sigma), \sigma\tau=\gamma}g(\tau)f(\sigma) = \sum_{\tau\nu=\gamma,\dom(\nu)=\ran(\nu)}g(\tau)\alpha_{\tau}(f(\nu)).
\end{equation}
But the right hand side of \eqref{eq:center.3} is the same as the right hand side of \eqref{eq:center.2} as $g(\lambda)\alpha_{\lambda}(f(\nu))=0$ unless $\dom(\nu)=\ran(\nu)$.

Conversely, suppose that $f\in Z(\Gamma_c(\mathscr G,\mathcal O))$.  We show that $f$ is a class function.  Suppose $\gamma\in \mathscr G\skel 1$ with $\dom(\gamma)\neq \ran(\gamma)$.  Let $U\subseteq \mathscr G\skel 0$ be compact open with $\dom(\gamma)\in U$ and $\ran(\gamma)\notin U$ (using that $\mathscr G\skel 0$ is Hausdorff).  Then $\chi_U\ast f(\gamma) = 0$ and $f\ast \chi_U(\gamma) = f(\gamma)\alpha_{\gamma}(1_{\dom(\gamma)}) = f(\gamma)$.  Thus $f(\gamma)=0$.  Next suppose that $\dom(\gamma)=\ran(\gamma)=x$ and $a\in \mathcal O_x$.  Then by sheaf theory over Hausdorff spaces with a basis of compact open sets, we can find a section $s\in \Gamma_c(\mathscr G\skel 0,\mathcal O)$ with $s(x) = a$ (see, for instance, Proposition~\ref{p:lots.of.secs} below or~\cite{Pierce}).  Then $s\ast f(\gamma) = s(x)f(\gamma) = af(\gamma)$ and $f\ast s(\gamma) = f(\gamma)\alpha_{\gamma}(s(x)) = f(\gamma)\alpha_{\gamma}(a)$.  This shows that $f$ satisfies the first condition in the definition of a class function.

Suppose now that $\ran(\sigma)=\dom(\gamma)=\ran(\gamma)$.  Choose a compact open bisection $U$ with $\sigma\in U$.  Then $f\ast \chi_U(\gamma\sigma) = f(\gamma)\alpha_{\gamma}(\chi_U(\sigma)) = f(\gamma)$.  On the other hand, since $\sigma$ is the only element of $U$ with $\ran(\sigma)=\ran(\gamma)$, we have that $\chi_U\ast f(\gamma\sigma) = \chi_U(\sigma)\alpha_{\sigma}(f(\sigma\inv \gamma\sigma)) = \alpha_{\sigma}(f(\sigma\inv \gamma\sigma))$.  This shows that $f(\gamma) = \alpha_{\sigma}(f(\sigma\inv \gamma\sigma)$, as required.  Thus $f$ is a class function.
\end{proof}

Our next result will be to show that $\Gamma_c(\mathscr G,\mathcal O)$ can be generated as a left $\Gamma_c(\mathscr G\skel 0,\mathcal O)$-module by characteristic functions belonging to a sufficiently large inverse semigroup of compact open bisections.

\begin{Prop}\label{p:generated}
Let $\mathscr G$ be an ample groupoid and $\mathcal O$ a sheaf of rings on $\mathscr G$.
Let $S\subseteq \mathscr G^a$ be an inverse semigroup satisfying (G1).
Then $\Gamma_c(\mathscr G,\mathcal O)$ is spanned by the functions supported on elements of $S$.
\end{Prop}
\begin{proof}
Set \[D=\{U\in \mathscr G^a\mid U\subseteq V,\ \text{for some}\ V\in S\}.\] We claim that $D$ is a basis for the topology on $\mathscr G\skel 1$. Let $\gamma\in \mathscr G\skel 1$.  Then a basis of neighborhoods of $\gamma$ is given by the compact open bisections  $U$ with  $\gamma\in U$.  There is $V\in S$ with $\gamma\in V$ by (G1).  As $U\cap V$ is a neighborhood of $\gamma$ and $\mathscr G^a$ is a basis for the topology, there is $W\in \mathscr G^a$ with $\gamma\in W$ and $W\subseteq U\cap V$.  But then $W\in D$, so we conclude that $D$ is a basis for the topology on $\mathscr G\skel 1$.

Let $R$ be the span of the functions supported on $S$.  Note that $R$ is a subring of $\Gamma_c(\mathscr G,\mathcal O)$ by Remark~\ref{r:conv}.
Suppose that $f\in \Gamma_c(\mathscr G,\mathcal O)$ is supported on $U\in D$.  If $U\subseteq V$ with $V\in S$, then $U$ is clopen in $V$ (as $V$ is Hausdorff and $U$ is compact). We may then view $f$ as supported on $V$ since $f|_V$ is continuous (because $U$ is clopen in $V$) and $f$ vanishes outside $V$.  Thus $f\in R$.

Next suppose that $U\in \mathscr G^a$ and $f$ is supported on $U$.  Since $D$ is a basis for the topology and $U$ is compact open, we have that $U=U_1\cup\cdots\cup U_n$ with $U_1,\ldots, U_n\in D$.  The $U_i$ are compact open and hence, since $U$ is Hausdorff, clopen in $U$.  Thus any finite intersection $V$ of the $U_i$ is compact open and belongs to $D$, as $D$ is a lower set.   Also $V$ is clopen in $U$.  Let $f_V$ be the mapping which agrees with $f$ on $V$ and is $0$ elsewhere.  Then $f_V\in \Gamma_c(\mathscr G,\mathcal O)$ and is supported on $V$.  Hence $f_V\in R$ by what we have already observed.

It follows from the principle of inclusion-exclusion that
\[f = \sum_{k=1}^n(-1)^{k-1}\sum_{\substack{I\subseteq \{1,\ldots, n\}\\ |I|=k}}f_{\bigcap_{i\in I} U_i}\in R.\]
This completes the proof.
\end{proof}

An important special case will be when $\mathscr G$ is a groupoid of germs of a Boolean action of an inverse semigroup $S$.  The proposition will imply that the functions supported on compact open bisections coming from $S$ span the ring.  This will be useful in expressing $\Gamma_c(\mathscr G,\mathcal O)$ as a skew inverse semigroup ring.

\section{Convolution algebras as quotients of skew inverse semigroup rings}\label{s:as.quotient}
In this section, we show that if $\mathscr G$ is an ample groupoid, $\mathcal O$ is a $\mathscr G$-sheaf of rings and $S\leq \mathscr G^a$ is an inverse subsemigroup satisfying (G1), then $\Gamma_c(\mathscr G,\mathcal O)$ is a quotient of $\Gamma_c(\mathscr G\skel 0, \mathcal O)\rtimes S$ for an appropriate spectral action of $S$ on $\Gamma_c(\mathscr G\skel 0,\mathcal O)$.  We shall later show that the quotient map is, in fact, an isomorphism if $S$ satisfies, in addition, (G2).  This occurs precisely when $\mathscr G$ is the groupoid of germs of the action of $S$ on $\mathscr G\skel 0$ (see~\cite{Exel}).  In particular, if $\mathcal O$ is a constant sheaf $\Delta(\mathscr G,R)$ for a commutative ring $R$ and $S=\mathscr G^a$, then we recover the ``Steinberg'' algebra $A_R(\mathscr G)$~\cite{Steinbergalgebra} as the skew inverse semigroup ring $C_c(\mathscr G\skel 0,R)\rtimes \mathscr G^a$ where $C_c(\mathscr G\skel 0, R)$ is the ring of locally constant mappings $\mathscr G\skel 0\to R$ with compact support.  In a later section, we shall realize skew inverse semigroup rings coming from spectral actions as convolution algebras for a sheaf of rings on an ample groupoid.

Let $\mathcal O=(E,p,\alpha)$ be a $\mathscr G$-sheaf of rings, which we fix for the rest of the section.  We shall continue to denote the stalk at $x\in \mathscr G\skel 0$ by $\mathcal O_x$. Let $S\leq \mathscr G^a$ be an inverse semigroup satisfying (G1).  Put $A=\Gamma_c(\mathscr G\skel 0,\mathcal O)$, which we view as the ring of compactly supported continuous sections of $\mathcal O$ over $\mathscr G\skel 0$ with pointwise operations.  We wish to define an action of $S$ on $A$.  If $U\subseteq \mathscr G\skel 0$ is open, we put $\mathcal O|_U = (p\inv(U),+,\cdot)$; it is a sheaf of rings on $U$.  Then $A(U)=\Gamma_c(U,\mathcal O|_U)$ can be identified with the subring of $A$ consisting of sections supported on $U$.  It is trivial to see that $A(U)$ is a two-sided ideal of $A$, for every open subset of $\mathscr G\skel 0$, as $\supp(fg)\subseteq \supp(f)\cap \supp(g)$.  For $s\in S$, put $D_s = A(\ran(s))$.  Note that since $\ran(s)$ is compact open, $D_s$ has an identity, the mapping $\chi_{\ran(s)}$.   We define an isomorphism $\til \alpha_s\colon D_{s^*}\to D_s$ by
\[\til \alpha_s(f)(\ran(\gamma)) = \alpha_{\gamma}(f(\dom(\gamma)))\] for $\gamma\in s$ and $f\in D_{s^*}$.

\begin{Prop}\label{p:action.ok}
The mapping $\til\alpha_s\colon D_{s^*}\to D_s$ is an isomorphism of rings.
\end{Prop}
\begin{proof}
First notice that, since $s$ is a bisection, there are no two $\gamma, \gamma'$ in $s$ with $\ran(\gamma) = \ran(\gamma')$. So the choice of $\gamma$ in the definition of $\til \alpha_s$ is unique. Next we check the continuity of $\til \alpha_s(f)$.  Clearly, $\gamma\mapsto \alpha(\gamma,f(\dom(\gamma)))=\til \alpha_s(f)(\ran(\gamma))$ is continuous from $s$ to $E$.  As $\ran|_s\colon s\to \ran(s)$ is a homeomorphism, it follows that $\til \alpha_s$ is continuous.  Note that $p(\til\alpha_s(f)(\ran(\gamma))) =p(\alpha_{\gamma}(\dom(\gamma))) = \ran(\gamma)$ and so $\til\alpha_s(f)$ is a section.   Thus $\til\alpha_s\colon D_{s^*}\to D_{s}$ is well defined.

Observe next that \[\til\alpha_{s^*}(\til\alpha_s(f))(\dom(\gamma))=\alpha_{\gamma\inv}(\til\alpha_s(f)(\ran(\gamma))) = \alpha_{\gamma\inv}\alpha_{\gamma}(f(\dom(\gamma))) = f(\dom(\gamma))\] for $\gamma\in s$ and so $\til\alpha_{s^*}\circ \til\alpha_s$ is the identity on $D_{s^*}$.  Interchanging the roles of $s^*$ and $s$ shows that $\til\alpha_s$ and $\til\alpha_{s^*}$ are inverse bijections.

Finally, we verify that $\til\alpha_s$ is a ring homomorphism.  Indeed, if $f,g\in D_{s^*}$ and $\gamma\in s$, then we have $\til\alpha_s(f+g)(\ran(\gamma)) = \alpha_{\gamma}(f(\dom(\gamma))+g(\dom(\gamma))) = \alpha_{\gamma}(f(\dom(\gamma)))+\alpha_{\gamma}(g(\dom(\gamma))) = \til\alpha_s(f)(\ran(\gamma))+\til\alpha_s(g)(\ran(\gamma))$.  Similarly, we have that $\til \alpha_s(fg)(\ran(\gamma))
= \til\alpha_s(f)(\ran(\gamma))\til\alpha_s(g)(\ran(\gamma))$ and so $\til\alpha_s$ is a ring homomorphism.
\end{proof}

Next we check that $\til\alpha\colon S\to I_A$ given by $\til\alpha(s) = \til\alpha_s$ is a homomorphism of inverse semigroups.

\begin{Prop}\label{p:action.ok2}
The mapping $\til\alpha\colon S\to I_A$ is a homomorphism.  Moreover, the action of $S$ on $A$ is non-degenerate and spectral.
\end{Prop}
\begin{proof}
To check that $\til\alpha$ is a homomorphism it suffices to check that $\til\alpha$ is order preserving, $\til \alpha(ef) = \til\alpha(e)\til\alpha(f)$ for $e,f\in E(S)$  and $\til\alpha(st) = \til\alpha(s)\til\alpha(t)$ whenever $s^*s=tt^*$.  If $s\leq t$, then $D_s\subseteq D_t$ and $D_{s^*}\subseteq D_{t^*}$ by definition.  If $f\in D_{s^*}$ and $\gamma\in s$, then $\gamma\in t$ and $\til\alpha_s(f)(\ran(\gamma)) = \alpha_{\gamma}(f(\dom(\gamma)))=\til\alpha_t(f)(\ran(\gamma))$.  Thus $\til\alpha(s)\leq \til\alpha(t)$.

 If $e\in E(S)$, then $D_e = A(e)$ and if $f\colon e\to E$ is a section, then $\til\alpha_e(f)(x) = \alpha_x(f(x)) = f(x)$ and so $\til\alpha_e$ is the identity on $D_e$.  Similarly, $D_f= A(f)$ and $\til\alpha_f$ is the identity on $D_f$.  So $\til\alpha_e\til\alpha_f$ is the identity on $D_e\cap D_f$, which consists of those functions supported on $e\cap f=ef$.  Thus $D_e\cap D_f=D_{ef}$ and $\til\alpha_{ef} = \til\alpha_e\til\alpha_f$.   Finally, if $s^*s=tt^*=e$, then $\til\alpha_s\colon D_e\to D_s$ and $\til\alpha_t\colon D_{t^*}\to D_e$.  But $(st)^*(st) = t^*t$ and $(st)(st)^* = ss^*$ and so $\til\alpha_{st}\colon D_{t^*}\to D_s$.  If $f\in D_{t^*}$ and $x\in \ran(s) = \ran(st)$, then we can find a unique $\gamma\in t$ and $\beta\in s$ with $\ran(\beta) =x$ and $\dom(\beta)=\ran(\gamma)$.  Then $\til\alpha_{st}(f)(x) = \alpha_{\beta\gamma}(f(\dom(\gamma))) = \alpha_{\beta}(\alpha_{\gamma}(f(\dom(\gamma))) = \alpha_{\beta}(\til \alpha_t(f)(\ran(\gamma))) = \til\alpha_s(\til\alpha_t(f))(x)$ and so $\til\alpha_{st} = \til\alpha_s\til\alpha_t$.

Since each $D_e$ with $e\in E(S)$ is unital, it remains to show that the action is non-degenerate.  By Proposition~\ref{p:generated} applied to $\mathscr G\skel 0$, viewed as an ample groupoid of identity arrows, we just need to show that each $x\in \mathscr G\skel 0$ belongs to some $e\in E(S)$.  By (G1), we have that $x\in s$ for some $s\in S$.  But then $x\in s^*s\in E(S)$, as was required.
This completes the proof.
\end{proof}

We can now form the skew inverse semigroup ring $A\rtimes S$.   In this section, we shall define a natural quotient map $A\rtimes S\to \Gamma_c(\mathscr G,\mathcal O)$.  For convenience, let us put $R=\Gamma_c(\mathscr G,\mathcal O)$ for the remainder of this section.  We define a covariant system $(R,\wh \theta,\wh \p)$.  The mapping $\wh\theta$ will just be the inclusion of $A=\Gamma_c(\mathscr G\skel 0,\mathcal O)$ into $R=\Gamma_c(\mathscr G,\mathcal O)$.  Define $\wh\p (s) = \chi_s$ for $s\in S$.

\begin{Prop}\label{p:is.covariant}
The triple $(R,\wh \theta,\wh \p)$ is a covariant system where $R=\Gamma_c(\mathscr G,\mathcal O)$.  Moreover, the induced homomorphism $\wh\theta\rtimes \wh \p\colon A\rtimes S\to R$ is surjective.
\end{Prop}
\begin{proof}
We first check the covariance axioms.  Clearly, $\wh\theta$ is a ring homomorphism.  Also $\wh \p(s)\wh \p(t) = \chi_s\ast \chi_t=\chi_{st}$ by Proposition~\ref{p:inverse.embed}.  Let $f\colon \mathscr G\skel 0\to E$ be a section supported on $\dom(s)$ with $s\in S$.  Then we compute $\chi_s\ast f\ast \chi_{s^*}$.  It will be supported on units $x$ in $s\dom(s)s^* = \ran(s)$.  For $x\in \ran(s)$, let $\gamma\in s$ with $\ran(\gamma)=x$.  Then we have that \[\chi_s\ast f\ast\chi_{s^*}(x) = \chi_s(\gamma)\alpha_{\gamma}(f(\dom(\gamma))\chi_{s^*}(\gamma\inv))=\alpha_{\gamma}(f(\dom(\gamma)))=\til\alpha_s(f)(x)\] as required.  Finally, if $e\in E(S)$, then $\chi_e$ is the identity of $A(e)$ and so $\wh \p(e) = \chi_e=\wh \theta(1_e)$.

Let us verify that $\wh \theta\rtimes \wh\p\colon A\rtimes S\to R$ is onto.  By Proposition~\ref{p:generated}, $R$ is generated by mappings $f$ supported on an element of $S$.  As observed earlier, any such function can be written in the form $g\chi_s$ where $g=f\circ (\ran|s)\inv \colon \ran(s)\to E$ is a continuous section, i.e., $g\in A(\ran(s)) = D_s$.  Then we have $(\wh\theta\rtimes \wh\p)(g\delta_s+\mathcal N) = g\ast \chi_s$.  But
\[g\ast \chi_s(\gamma) = \begin{cases}g(\ran(\gamma)),  & \text{if}\ \gamma\in s\\ 0, & \text{else.}\end{cases}\]  Thus $g\ast\chi_s = g\chi_s=f$.  This completes the proof.
\end{proof}

\section{The disintegration theorem}\label{desintegration}
Let $\mathscr G$ be an ample groupoid and $\mathcal O=(E,p,\alpha)$ a $\mathscr G$-sheaf of rings.  Put $R=\Gamma_c(\mathscr G,\mathcal O)$; it is a ring with local units.  A (left) $R$-module $M$ is \emph{unitary} if $RM=M$.  We denote by $\module{R}$ the category of unitary (left) $R$-modules.  Our goal in this section is to generalize the disintegration theorem of~\cite{groupoidbundles} and prove that $\module{R}$ can be identified with the category of $\mathscr G$-sheaves of $\mathcal O$-modules.  In the next section we will show that these categories correspond to unitary covariant systems for $A=\Gamma_c(\mathscr G\skel 0,\mathcal O)$ and $S$, where $S\leq \mathscr G^a$ satisfies the germ conditions.  This will then be used to complete the proof that $A\rtimes S\cong R$.

A \emph{$\mathscr G$-sheaf of $\mathcal O$-modules} $\mathcal M=(F,q,\beta)$ is a $\mathscr G$-sheaf such that each stalk $\mathcal M_x$ has a (unitary) left $\mathcal O_x$-module structure such that:
\begin{enumerate}
  \item [(SM1)] addition $+\colon F\times_{q,q} F\to F$ is continuous;
  \item [(SM2)] the  module action $E\times_{p,q} F\to F$ is continuous;
  \item [(SM3)] $\beta_{\gamma}(rm) = \alpha_{\gamma}(r)\beta_{\gamma}(m)$ for all $r\in \mathcal O_{\dom(\gamma)}$ and $m\in \mathcal M_{\dom(\gamma)}$.
\end{enumerate}

The condition (SM3) basically says that $\beta_{\gamma}\colon \mathcal M_{\dom(\gamma)}\to \mathcal M_{\ran(\gamma)}$ is a module homomorphism once we identify $\mathcal O_{\dom(\gamma)}$ and $\mathcal O_{\ran(\gamma)}$ via the isomorphism $\alpha_{\gamma}$. One can again prove that the zero section and negation are continuous~\cite{Dowker}.   Note that $\mathcal M$ is a sheaf of $\mathcal O$-modules on $\mathscr G\skel 0$ in the classical sense of sheaf theory.  A morphism of $\mathscr G$-sheaves of $\mathcal O$-modules is a morphism of $\mathscr G$-sheaves $h\colon \mathcal M\to \mathcal N$ that restricts to an $\mathcal O_x$-module homomorphism $\mathcal M_x\to \mathcal N_x$ for all $x\in \mathscr G\skel 0$.

We shall use the following proposition without comment.

\begin{Prop}\label{p:lots.of.secs}
Let $X$ be a Hausdorff space with a basis of compact open sets and let $\mathcal A=(F,q)$ be a sheaf of abelian groups on $X$ (i.e., a sheaf of $\Delta(\mathbb Z)$-modules).    Then, for each $f\in \mathcal A_x$, there is a section $s\colon X\to F$ with compact support such that $s(x)=f$.
\end{Prop}
\begin{proof}
Since $q$ is a local homeomorphism and $X$ has a basis of compact open sets, so does $F$.  Choose a compact open neighborhood $W$ of $f$ such that $q|_W\colon W\to q(W)$ is a homeomorphism.  Define $s\colon X\to F$ by
\[s(y) = \begin{cases}  (q|_W)\inv (y), & \text{if}\ y\in q(W)\\ 0_y, & \text{else.}\end{cases}\]   Note that $q(W)$ is compact open and hence clopen since $X$ is Hausdorff.  Thus $s$ is continuous as its restriction to the open set $q(W)$ is $(q|_W)\inv$ and its restriction to the disjoint open set $X\setminus q(W)$ is the zero section, which is continuous.  Since $x=q(f)\in q(W)$, we deduce that $s(x) = (q|_W)\inv(q(f)) = f$.  The support of $s$ is contained in the compact set $q(W)$ and hence is compact (being closed).  This completes the proof.
\end{proof}

If $\mathcal M$ is a $\mathscr G$-sheaf of $\mathcal O$-modules, then we can look at the set $M=\Gamma_c(\mathscr G,\mathcal M)$ of continuous sections $s\colon \mathscr G\skel 0\to F$ with compact support. This is an abelian group with pointwise operations.  We define an $R$-module structure on it by putting, for $f\in R$ and $m\in M$,
\[(fm)(x) = \sum_{\gamma\in \ran\inv(x)}f(\gamma)\beta_{\gamma}(m(\dom(\gamma))).\]
Since $f$ is non-zero on only finitely many elements of $\ran\inv(x)$ by Proposition~\ref{p:finiteness.prop}, this sum is finite.   To check that $fm$ is continuous with compact support it suffices to consider the case when $f=s\chi_U$ with $U\in \mathscr G^a$ and $s\colon \ran(U)\to E$ a section, as these generate $\Gamma_c(\mathscr G,\mathcal O)$. Let $h=(\ran|_U)\inv\colon \ran(U)\to U$.  Then \[(s\chi_Um)(x) = \begin{cases}s(x)\beta_{h(x)}(m(\dom(h(x)))), & \text{if}\ x\in \ran(U)\\ 0_x, & \text{else}\end{cases}\] which is continuous with compact support as $\ran(U)$ is compact open, $\mathscr G\skel 0$ is Hausdorff and $s$, $\beta$, $m$, $h$, and $\dom$ are continuous.

 If $r\colon \mathcal M\to \mathcal N$ is a morphism of sheaves of $\mathcal O$-modules, then $s\mapsto r\circ s$ is an $R$-module homomorphism $\Gamma_c(\mathscr G,\mathcal M)\to \Gamma_c(\mathscr G,\mathcal N)$.

\begin{Prop}\label{p:is.unitary}
The construction $\mathcal M\longmapsto \Gamma_c(\mathscr G,\mathcal M)$ is a functor from the category of $\mathscr G$-sheaves of $\mathcal O$-modules to the category of $\module{\Gamma_c(\mathscr G,\mathcal O)}$.
\end{Prop}
\begin{proof}
Most of the proof is a straightforward adaptation of the arguments in~\cite{groupoidbundles} for the case where $\mathcal O$ is a constant sheaf of commutative rings.  We check here that $\Gamma_c(\mathscr G,\mathcal M)$ is a unitary $R$-module; functoriality is easy to check.  The most difficult detail is that $f(gm) = (f\ast g)m$ for $f,g\in \Gamma_c(\mathscr G,\mathcal O)$ and $m\in \Gamma_c(\mathscr G,\mathcal M)$, the remaining tedious details that $\Gamma_c(\mathscr G,\mathcal M)$ is an $R$-module are left  to the reader.   We compute that
\begin{align*}
(f(gm))(x) &=\sum_{\sigma\in \ran\inv(x)} f(\sigma)\beta_{\sigma}((gm)(\dom(\sigma))) \\ &= \sum_{\sigma\in \ran\inv(x)}\sum_{\tau\in \ran\inv (\dom(\sigma))}f(\sigma)\beta_{\sigma}(g(\tau)\beta_{\tau}(m(\dom(\tau))))
\\ &=\sum_{\sigma\in \ran\inv(x)}\sum_{\tau\in \ran\inv (\dom(\sigma))}f(\sigma)\alpha_{\sigma}(g(\tau))\beta_{\sigma}(\beta_{\tau}(m(\dom(\tau))))
\\ &=\sum_{\sigma\in \ran\inv(x)}\sum_{\tau\in \ran\inv (\dom(\sigma))}f(\sigma)\alpha_{\sigma}(g(\tau))\beta_{\sigma\tau}(m(\dom(\sigma\tau)))\\ &= ((f\ast g)m)(x)
\end{align*}

We briefly check that $\Gamma_c(\mathscr G,\mathcal M)$ is unitary.  Let $m\in \Gamma_c(\mathscr G,\mathcal M)$.  Then since $m$ has compact support, we can find a compact open set $U\subseteq \mathscr G\skel 0$ containing the support of $m$.  Then $\chi_Um =m$, as one immediately verifies.  Thus $\Gamma_c(\mathscr G,\mathcal M)$ is unitary.
\end{proof}

Now we present a functor from $\module{R}$ to the category of $\mathscr G$-sheaves of $\mathcal O$-modules that will be quasi-inverse to the functor $\mathcal M\mapsto \Gamma_c(\mathscr G,\mathcal M)$.  The details are very similar to those in~\cite{groupoidbundles} and so we highlight what is different.
If $M$ is a unitary $R$-module and $x\in \mathscr G\skel 0$, put \[M_x = \varinjlim_{x\in U} \chi_UM\] where the direct limit runs over all compact open neighborhoods $U$ of $x$ in $\mathscr G\skel 0$.  Here, if $U\subseteq V$, then $\chi_V\ast \chi_U=\chi_U=\chi_U\ast \chi_V$ and so we have a restriction homomorphism $\rho^V_U\colon \chi_VM\to \chi_UM$ of abelian groups given by $\rho^V_U(m) = \chi_Um$.  The homomorphisms $\rho^V_U$, as $U,V$ vary, clearly form a directed system.  Thus $M_x$ is an abelian group.  We provide an alternative description that is convenient for computations.   Let \[N_x = \{m\in M\mid \chi_Um=0\ \text{for some}\ U\subseteq \mathscr G\skel 0\ \text{compact open with}\ x\in U\}.\]  Then $N_x$ is an additive subgroup of $M$ for if $\chi_Um=0$ and $\chi_Vn=0$ with $x\in U,V$, then $\chi_{U\cap V}(m+n) = \chi_{U\cap V}\ast \chi_Um+\chi_{U\cap V}\ast \chi_Vn=0$ and $x\in U\cap V$.
We claim that $M/N_x\cong M_x$.
The isomorphism is induced by sending $m+N_x$ to the class of $\chi_Um$ in $M_x$, where $x\in U\subseteq \mathscr G\skel 0$  compact open.  We write $[m]_x$ for the class of $m\in M$ in $M_x$, which we identify with $M/N_x$ from now on.  Note that if $m\in M$ and $U\subseteq \mathscr G\skel 0$ is compact open neighborhood of $x$, then $m-\chi_Um\in N_x$ and so $[m]_x=[\chi_Um]_x$.  Moreover, $[m]_x=[n]_x$ if and only if there is a compact open neighborhood $U$ of $x$ in $\mathscr G\skel 0$ with $\chi_Um=\chi_Un$.

Define \[F= \coprod_{x\in \mathscr G\skel 0} M_x\] and $q\colon F\to \mathscr G\skel 0$ by $q([m]_x) = x$.  Put a topology on $F$ by taking as a basis all sets of the form
\[D(m,U) = \{[m]_x\mid x\in U\},\] where $m\in M$ and $U\subseteq \mathscr G\skel 0$ is compact open.  It is easy to check that $q\colon D(m,U)\to U$ is a homeomorphism.  To define the $\mathscr G$-sheaf structure, for $\gamma\in \mathscr G$, put \[\beta_{\gamma}([m]_{\dom(\gamma)}) = [\chi_Um]_{\ran(\gamma)}\] where $U$ is any compact open bisection containing $\gamma$.  It is a straightforward adaptation of the proof in~\cite{groupoidbundles} that $\beta\colon \mathscr G\skel 1\times_{\dom,q} F\to F$ is well defined and continuous, and turns $$\mathrm{Sh}(M) = (F,q,\beta)$$ into a $\mathscr G$-sheaf.  For instance, to see that $\beta_{\gamma}$ is independent of the choice of $U$, suppose that $\gamma\in U,V$ with $U,V\in \mathscr G^a$.  Then, as $\mathscr G^a$ is a basis for the topology on $\mathscr G\skel 1$, we can find $W\in \mathscr G^a$ with $\gamma\in W\subseteq U\cap V$.  Then $WW\inv U = W=WW\inv V$ and $\ran(\gamma)\in WW\inv$.  Thus we have that
$[\chi_Um]_{\ran(\gamma)} = [\chi_{WW\inv}(\chi_Um)]_{\ran(\gamma)} = [\chi_Wm]_{\ran(\gamma)} = [\chi_{WW\inv}(\chi_Vm)]_{\ran(\gamma)} = [\chi_Vm]_{\ran(\gamma)}$.  Similarly, if we fix $U\in \mathscr G^a$ with $\gamma\in U$ and if $[m]_{\dom(\gamma)} = [n]_{\dom(\gamma)}$, then we can find $V\subseteq \mathscr G\skel 0$ compact open with $\dom(\gamma)\in V$ and $\chi_Vm=\chi_Vn$.  Then we have that $\gamma\in UV$ and so, by the previous verification, we have that $[\chi_Um]_{\ran(\gamma)} = [\chi_{UV}m]_{\ran(\gamma)} = [\chi_U(\chi_Vm)]_{\ran(\gamma)} = [\chi_U(\chi_Vn)]_{\ran(\gamma)} =[\chi_{UV}n]_{\ran(\gamma)}=[\chi_Un]_{\ran(\gamma)}$ and so $\beta_{\gamma}$ is well defined.

 Let us verify that $\mathrm{Sh}(M)$ is a $\mathscr G$-sheaf of $\mathcal O$-modules.  It is clearly a $\mathscr G$-sheaf of abelian groups with respect to the fiberwise addition $[m]_x+[n]_x=[m+n]_x$ (i.e., addition is fiberwise continuous and each $\beta_\gamma$ is an additive homomorphism).

To define the $\mathcal O_x$-module structure on $M_x$, let $r\in \mathcal O_x$.  We can choose a section $t\in \Gamma_c(\mathscr G\skel 0,\mathcal O)$ with $t(x) = r$ (using Proposition~\ref{p:lots.of.secs}).  We define $r[m]_x = [tm]_x$.  This is independent of the choice of $t$ and $m$.   Let us first consider the case of $m$.  If $[m]_x=[n]_x$, then there is a compact open neighborhood $W$ of $x$ with $\chi_Wm=\chi_Wn$.  Also $t\ast \chi_W =\chi_W\ast t$,  as both sections agree with $t$ on $W$ and are zero outside of $W$. Then $[tm]_x = [\chi_W\ast tm_X]_x = [t\ast \chi_Wm]_x=[t\ast \chi_Wn]_x= [\chi_W\ast tn]_x=[tn]_x$, yielding independence of the choice of $m$.
If $s$ is a section with $s(x)=r=t(x)$, then since the zero section has open image, the set of points where $t-s$ is zero is open and contains $x$.  Hence there is a compact open neighborhood $U$ of $x$ with $t|_U=s|_U$.  Consequently, $\chi_U\ast s=\chi_U\ast t$ and so $[tm]_x=[\chi_U\ast tm]_x = [\chi_U\ast sm]_x = [sm]_x$.

Now let $\gamma\in \mathscr G\skel 1$ and $r\in \mathcal O_{\dom(\gamma)}$. Let $U$ be a compact open bisection containing $\gamma$ and $t$ a compactly supported section with $t(\dom(\gamma))=r$. Then $\dom(\gamma)\in \dom(U)$ and so $[m]_{\dom(\gamma)} = [\chi_{U\inv}\ast \chi_Um]_{\dom(\gamma)}$ for any $m\in $M.  Thus we have that
\[\beta_{\gamma}(r[m]_{\dom(\gamma)}) = [\chi_U\ast t\ast \chi_{U\inv} (\chi_Um)]_{\ran(\gamma)}.\]  Now $\chi_U\ast t\ast \chi_{U\inv}\in \Gamma_c(\mathscr G\skel 0,\mathcal O)$ and satisfies $\chi_U\ast t\ast \chi_{U\inv}(\ran(\gamma)) = \alpha_{\gamma}(t(\dom(\gamma))) =\alpha_{\gamma}(r)$.  We thus obtain that
\[\beta_{\gamma}(r[m]_{\dom(\gamma)}) = \alpha_{\gamma}(r)[\chi_Um]_{\ran(\gamma)} = \alpha_{\gamma}(r)\beta_{\gamma}([m]_{\dom(\gamma)}),\] as required.  The remaining verifications that $\mathrm{Sh}(M)$ is a $\mathscr G$-sheaf of $\mathcal O$-modules are analogous to the special case in~\cite{groupoidbundles}.   If $f\colon M\to N$ is a homomorphism of unitary $R$-modules, then it is straightforward that $[m]_x\mapsto [f(m)]_x$ is a morphism of $\mathscr G$-sheaves of $\mathcal O$-modules.
We summarize in the following proposition.

\begin{Prop}\label{p:mod.to.sheaf}
The construction $M\longmapsto \mathrm{Sh}(M)$ is a functor from the category $\module{\Gamma_c(\mathscr G,\mathcal O)}$ to the category of $\mathscr G$-sheaves of $\mathcal O$-modules.
\end{Prop}

Most of the rest of this section is concerned with proving that there are natural isomorphisms $M\cong \Gamma_c(\mathscr G,\mathrm{Sh}(M))$ and $\mathcal M\cong \mathrm{Sh}(\Gamma_c(\mathscr G,\mathcal M))$.  The proofs are again very similar to those in~\cite{groupoidbundles} and so we skip many of the tedious details.

\begin{Thm}[Disintegration theorem]\label{t:disint}
The functors $M\longmapsto \mathrm{Sh}(M)$ and $\mathcal M\longmapsto \Gamma_c(\mathscr G,\mathcal O)$ provide an equivalence between the category $\module{\Gamma_c(\mathscr G,\mathcal O)}$ of unitary $\Gamma_c(\mathscr G,\mathcal O)$-modules and the category of $\mathscr G$-sheaves of $\mathcal O$-modules.
\end{Thm}
\begin{proof}
Put $R=\Gamma_c(\mathscr G,\mathcal O)$.  First we start with a unitary $R$-module $M$ and define the isomorphism $\eta\colon M\to \Gamma_c(\mathscr G,\mathrm{Sh}(M))$, which the reader can verify is natural in $M$.  If $m\in M$, then we define a section $\wh m\colon \mathscr G\skel 0\to \coprod_{x\in \mathscr G\skel 0}M_x$ by $\wh m(x) = [m]_x$.   Let us check continuity.  Let $(n,U)$ with $n\in M$ and $U$ a compact open neighborhood of $x$ be a basic neighborhood of $\wh m(x)$.  Then $[m]_x=\wh m(x) = [n]_x$ and so we can find $W$ a compact open neighborhood of $x$ with $\chi_Wm=\chi_Wn$.  Then $x\in U\cap W$ and if $y\in U\cap W$, then $\wh m(y) = [m]_y = [\chi_Wm]_y = [\chi_Wn]_y = [n]_y$ and so $\wh m(U\cap W)\subseteq (n,U)$.  Thus $\wh m$ is continuous.  Since $M$ is unitary, we can find $U\subseteq \mathscr G\skel 0$ compact open with $\chi_Um=m$.  Suppose that $y\notin U$.  Then since $\mathscr G\skel 0$ is Hausdorff and $U$ is compact, we can find a neighborhood $V$ of $y$ disjoint from $U$.  Since $\mathscr G\skel 0$ has a basis of compact open sets, we may assume that $V$ is compact open.  Then $[m]_y=[\chi_Vm]_y = [\chi_V\ast \chi_Um]_y = [\chi_{U\cap V}m]_y = [0]_y$ as $\chi_\emptyset$ is the zero of $R$.   Thus $\wh m$ has support contained in the compact open set $U$ and hence $\wh m$ has compact support (as the support of a section is always closed).

It is straightforward to verify that $m\mapsto \wh m$ is an abelian group homomorphism $\eta\colon M\to \Gamma_c(\mathscr G, \mathrm{Sh}(M))$.  To show that it is an $R$-module homomorphism, it is enough to verify  $\eta(fm)=f\eta(m)$ (i.e., $\wh{fm}=f\wh m$) on additive generators $f=g\chi_U$ of $R$, where $U$ is a compact open bisection and $g$ is a section of $\mathcal O$ supported on $\ran(U)$.   Since $\chi_{\ran(U)}f=f$, it follows that $f\wh m$ is supported on $\ran(U)$.  Also, it follows  from the proof in the  previous paragraph that the support of $\wh{fm}$ is contained in $\ran(U)$ as $\chi_{\ran(U)}fm=fm$.  Let $x\in \ran(U)$ and let $\gamma\in U$ with $\ran(\gamma) = x$.  Recall that $f=g\chi_U=g\ast \chi_U$ (cf., the proof of Proposition~\ref{p:is.covariant}).  Then $\wh{fm}(x) = [fm]_x= [g\ast \chi_Um]_x =g(x)[\chi_Um]_x= g(x)\beta_{\gamma}([m]_{\dom(\gamma)})$.  On the other hand
\[f\wh m(x) = g(x)\beta_{\gamma}(\wh m(\dom(\gamma))) = g(x)\beta_{\gamma}([m]_{\dom(\gamma)})=\wh{fm}(x), \] as required.

To see that $\eta$ is injective, suppose that $\eta(m)=\wh m$ is the zero section.  We show that $m=0$.  Let $U\subseteq \mathscr G\skel 0$ compact open be arbitrary.  We show that $\chi_Um=0$.  Since $M$ is unitary, it will then follow that $m=0$.  For each $x\in U$, we have that $[0]_x=\wh m(x) = [m]_x$ and so we can find a compact open neighborhood $V_x$ of $x$ in $\mathscr G\skel 0$ with $\chi_{V_x}m=0$.    Since $U$ is compact, we can cover it with a finite set $V_1,\ldots, V_n$ of compact open sets with $\chi_{V_i}m=0$ for $i=1,\ldots, n$.  Then $V=V_1\cup\cdots \cup V_n$ is compact open and $\chi_V$ is an integral linear combination of $\chi_{V_1},\ldots,\chi_{V_n}$ by the principle of inclusion-exclusion.  Thus $\chi_Vm=0$.  Therefore, $\chi_Um = \chi_{U\cap V}m = \chi_U\ast\chi_Vm = 0$.  We conclude that $m=0$ and so $\eta$ is injective.

Finally, we must prove that $\eta$ is surjective.  Let $t\in \Gamma_c(\mathscr G,\mathrm{Sh}(M))$.  For each $x\in \mathscr G\skel 0$, let  $(m_x,U_x)$ be a basic neighborhood of $t(x)$.  Then, by continuity, we can find a compact open neighborhood $W_x$ of $x$ with $t(W_x)\subseteq (m_x,U_x)$.  It then follows that $t(y) = [m_x]_y$ for all $y\in W_x$.  Using that the support of $t$ is compact, we can obtain a finite collection of compact open neighborhoods $W_1,\ldots, W_n$ and $m_1,\ldots, m_n\in M$ such that $\supp(t)\subseteq W_1\cup \cdots \cup W_n$ and $t(y) = [m_i]_y$ for all $y\in W_i$.  Put $V_1=W_1$ and, more generally, $V_i = W_i\setminus (V_1\cup\cdots\cup V_{i-1})$.  Then $V_i\subseteq W_i$, for all $i$, the $V_i$ are pairwise disjoint and $V_1\cup\cdots \cup V_n = W_1\cup\cdots \cup W_n$.  Consequently, $t(y) = [m_i]_y$ for all $y\in V_i$.  Put \[m = \chi_{V_1}m_1+\cdots+\chi_{V_n}m_n.\]  We claim that $t=\wh m=\eta(m)$.  Let $V=V_1\cup\cdots\cup V_n$.  Note that $\supp(t)\subseteq V$ and $\chi_Vm = m$.  Therefore, $\supp(\wh m)\subseteq V$ by the first paragraph of the proof.  Thus it remains to show that $t(x)=\wh m(x)$ for all $x\in V$.  There is a unique index $i$ with $x\in V_i$ and $t(x) = [m_i]_x$.  On the other hand
\begin{align*}
\wh m(x) &= [m]_x = [\chi_{V_i}m]_x =[\chi_{V_i}(\chi_{V_1}m_1+\cdots +\chi_{V_n}m_n)]_x = [\chi_{V_i}m_i]_x = [m_i]_x\\ &=t(x)
\end{align*}
 as $\chi_{V_i}\ast \chi_{V_j}=0$ whenever $i\neq j$.  Thus $\eta$ is surjective.  This completes the proof that $M\cong \Gamma_c(\mathscr G,\mathrm{Sh}(M))$.

Now we prove that if $\mathcal M=(F,q,\beta)$ is a $\mathscr G$-sheaf of $\mathcal O$-modules, then $\mathcal M\cong \mathrm{Sh}(\Gamma_c(\mathscr G,\mathcal M))$.  We leave the reader to check naturality.  We define $v\colon \mathrm{Sh}(\Gamma_c(\mathscr G,\mathcal M))\to \mathcal M$ by $v([t]_x) = t(x)$ for $t\in \Gamma_c(\mathscr G,\mathcal M)$.  This is well defined because if $[s]_x=[t]_x$, then there exists a compact open neighborhood $U$ of $x$ in $\mathscr G\skel 0$ with $\chi_Us =\chi_Ut$ and so $s(x)=(\chi_Us)(x)=(\chi_Ut)(x)=t(x)$.  It is a straightforward adaptation of the argument in~\cite{groupoidbundles} to check that $v$ is continuous and a morphism of $\mathscr G$-sheaves.  For example, to check continuity of $v$, let $U$ be a neighborhood of $v([t]_x) = t(x)$.  Then we can find a compact open neighborhood $W$ of $x$ with $t(W)\subseteq U$.  Then $(t,W)$ is a neighborhood of $[t]_x$ and if $[t]_y\in (t,W)$, then $y\in W$ and so $v([t]_y) = t(y)\in U$.  Thus $v((t,W))\subseteq U$, yielding continuity.  To see that it is a $\mathscr G$-sheaf morphism let $\gamma\colon x\to y$ be an arrow of $\mathscr G$ and let $U\in \mathscr G^a$ with $\gamma\in U$.  Write $\beta'$ for the action of $\mathscr G$ on $\mathrm{Sh}(\Gamma_c(\mathscr G,\mathcal M))$.  Then if $t\in \Gamma_c(\mathscr G,\mathcal M)$, we have $\beta'_{\gamma}([t]_x) = [\chi_Ut]_y$.  Therefore, $v(\beta'_{\gamma}([t]_x)) = v([\chi_Ut]_y) = (\chi_Ut)(y) = \chi_U(\gamma)\beta_{\gamma}(t(x)) = \beta_{\gamma}(v([t]_X))$, as required.

We check that $v$ is an isomorphism of $\mathscr G$-sheaves of $\mathcal O$-modules.  First, let $r\in \mathcal O_x$ and choose a section $f\in \Gamma_c(\mathscr G\skel 0,\mathcal O)$ with $f(x)=r$.  Then $v(r[t]_x) = v([ft]_x) = (ft)(x) = f(x)t(x) = rv([t]_x)$ and so $v$ is a morphism.  Since morphisms of $\mathscr G$-sheaves are automatically local homeomorphisms, it remains to show that $v$ is bijective.  Let $m\in \mathcal M_x$.  Let $t\in \Gamma_c(\mathscr G,\mathcal M)$ be a section with $t(x) = m$ (such exists by Proposition~\ref{p:lots.of.secs}).  Then $v([t]_x) = t(x)=m$ and so $v$ is onto.   Suppose that $v([t]_x) = 0_x$, that is, $t(x)=0_x$.  Then since the image of the zero section is open, we can find a compact open neighborhood $U$ of $x$ with $t(U)=0$.  Then $[t]_x = [\chi_U t]_x = [0]_x$ as $\chi_U t$ is zero on $\mathscr G\skel 0$.   This completes the proof that $v$ is an isomorphism.
\end{proof}

Note that in applications we shall mostly use the isomorphism $M\cong \Gamma_c(\mathscr G,\mathrm{Sh}(M))$ since we want to disintegrate representations of $R$ to representations of $\mathscr G$.

\section{Groupoid algebras as skew inverse semigroup rings}\label{gpdasisr}

Our next goal is to show that unitary $A\rtimes S$-modules also can be represented as modules of global sections of $\mathscr G$-sheaves of $\mathcal O$-modules, where $A=\Gamma_c(\mathscr G\skel 0,\mathcal O)$ and $S\leq \mathscr G^a$ is an inverse subsemigroup satisfying  the germ conditions.
This will allow us to prove that the surjective homomorphism in Proposition~\ref{p:is.covariant} is, in fact, an isomorphism in this case.  Here, we form the skew inverse semigroup ring via the action in Proposition~\ref{p:action.ok2}.  So from now on assume that $S\leq \mathscr G^a$ is an inverse semigroup satisfying the germ conditions.

By Theorem~\ref{t:adjunction}, a unitary $A\rtimes S$-module $M$ is the same thing as an abelian group $M$ and a covariant system $(\mathrm{End}_{\mathbb Z}(M),\theta,\p)$ with the extra condition that $M$ is a unitary $A$-module under the module action $am=\theta_a(m)$, where we put $\theta(a)=\theta_a$ and $\p(s)=\p_s$.  This latter condition is equivalent to asking that, for each $m\in M$, there is a compact open $U\subseteq \mathscr G\skel 0$ with $\chi_Um=m$.

Let us recall that if $\psi\colon R\to S$ is a surjective ring homomorphism, then each $S$-module $N$ becomes an $R$-module, called the \emph{inflation} of $N$ along $\psi$, by putting $rn=\psi(r)n$ for $r\in R$ and $n\in N$; inflation embeds the category of $S$-modules as a full subcategory of the category of $R$-modules.

To each unitary $A\rtimes S$-module $M$, we will associate a $\mathscr G$-sheaf $\mathcal M=(F,q,\beta)$ of $\mathcal O$-modules such that $M\cong \Gamma_c(\mathscr G,\mathcal M)$ as $A\rtimes S$-modules, where we view $\Gamma_c(\mathscr G,\mathcal M)$ as an $A\rtimes S$-module via inflation along the surjective homomorphism \[\wh \theta\rtimes \wh \p\colon A\rtimes S\to \Gamma_c(\mathscr G,\mathcal O)\] from Proposition~\ref{p:is.covariant}.  In other words, we shall show that every unitary $A\rtimes S$-module action factors through $\wh\theta\rtimes \wh \p$.  Applying this to the regular module, which is unitary and faithful, will allow us to prove that $\wh \theta\rtimes \wh p$ is an isomorphism.

Since $M$ is an $A$-module, applying Theorem~\ref{t:disint} to $\mathscr G\skel 0$, we obtain a $\mathscr G\skel 0$-sheaf of $\mathcal O$-modules $\mathscr M$ with underlying space $F=\coprod_{x\in \mathscr G\skel 0} M_x$ and $M\cong \Gamma_c(\mathscr G\skel 0,\mathcal M)$ as an $A$-module, as in the construction of $\mathrm{Sh}(M)$ in the previous section.  Note that the action of $r\in \mathcal O_x$ on $[m]_x$ is given by $[\theta_a(m)]_x$, where $a\in A$ is a section with $a(x)=r$.

 We now add a $\mathscr G$-sheaf structure as follows.  Let $\gamma\in \mathscr G\skel 1$ and choose $s\in S$ with $\gamma\in s$; such exists by (G1).  Put \[\beta_{\gamma}([m]_{\dom(\gamma)})=[\p_s(m)]_{\ran(\gamma)}\] to define $\beta$.  Let us check that $\beta_{\gamma}$ is well defined.  First fix $m$ and suppose that $\gamma\in t$ with $t\in S$. By (G2), there exists $u\in S$ with $\gamma\in u$ and $u\leq s,t$.  Put $U=\ran(u)=uu^*$ and note that $Us=u=Ut$ and $\ran(\gamma)\in U$.    Then, using that $(\mathrm{End}_{\mathbb Z}(M),\theta,\p)$ is a covariant system, we have that
\[[\p_s(m)]_{\ran(\gamma)} = [\theta_{\chi_U}\p_s(m)]_{\ran(\gamma)} = [\p_{uu^*}\p_s(m)]_{\ran(\gamma)}= [\p_u(m)]_{\ran(\gamma)}.\]
A similar argument shows that $[\p_t(m)]_{\ran(\gamma)} = [\p_u(m)]_{\ran(\gamma)}$ and so $\beta_{\gamma}$ does not depend on the choice of $s$.
On the other hand, if $[m]_{\dom(\gamma)} = [n]_{\dom(\gamma)}$ and $\gamma\in s$, then we can find a compact open neighborhood $W$ of $\dom(\gamma)$ with $\theta_{\chi_W}(m)=\theta_{\chi_W}(n)$.  Moreover, we may shrink $W$ so that $W\subseteq \dom(s)$.  Then $\chi_W\in D_{s^*}$ and so, by covariance, we have that $\p_s \theta_{\chi_W}\p_{s^*} = \theta_{\til\alpha_s(\chi_W)} = \theta_{\chi_{\ran(sW)}}$ where the last equality follows because if $\tau\in s$, then \[\til\alpha_s(\chi_W)(\ran(\tau))=\alpha_{\tau}(\chi_W(\dom(\tau))) =\begin{cases} 1_{\ran(\tau)}, & \text{if}\ \dom(\tau)\in W\\ 0, & \text{else.}\end{cases} \]  Therefore, since $\ran(\gamma)\in \ran(sW)$, we have that
\begin{align*}
[\p_s(m)]_{\ran(\gamma)} & = [\theta_{\chi_{\ran(sW)}}\p_s(m)]_{\ran(\gamma)}= [\p_s\theta_{\chi_W}\p_s^*\p_s(m)]_{\ran(\gamma)} \\ & = [\p_s\theta_{\chi_W}\theta_{\chi_{\dom(s)}}(m)]_{\ran(\gamma)} = [\p_s\theta_{\chi_W}(m)]_{\ran(\gamma)}.
\end{align*}
  Similarly, we have $[\p_s(n)]_{\ran(\gamma)} = [\p_s\theta_{\chi_W}(n)]_{\ran(\gamma)}$. As   $\theta_{\chi_W}(m)=\theta_{\chi_W}(n)$, it follows that $[\p_s(m)]_{\ran(\gamma)} = [\p_s(n)]_{\ran(\gamma)}$, as required.

It is straightforward to verify that $\beta$ makes $\mathcal M$ into a $\mathscr G$-sheaf.  To check continuity of $\beta$, let $\gamma\colon x\to y$, $m\in M$ and let $s\in S$ with $\gamma\in s$.  Then $\beta_{\gamma}([m]_x) = [\p_s(m)]_y$ and so a basic neighborhood of $\beta_{\gamma}([m]_x)$ is of the form $(\p_s(m),W)$ with $y\in W$ compact open and $W\subseteq \ran(s)$.  Consider the neighborhood $Z=(Ws\times (m,\dom(Ws)))\cap (\mathscr G\skel 1\times_{\dom,q}F)$.  Firstly, $(\gamma,[m]_x)\in Z$.  A typical element of $Z$ is of the form $(\gamma',[m]_z)$ with $\gamma'\in Ws$ and $\dom(\gamma')=z$.  Then $\gamma'\in s$ and so $\beta_{\gamma'}([m]_z) = [\p_s(m)]_{\ran(\gamma')}$ and $\ran(\gamma')\in \ran(Ws) = W$.  Thus $\beta(Z)\subseteq (\p_s(m),W)$, yielding continuity of $\beta$.  If $x\in \mathscr G\skel 0$, then we can choose $s\in S$ with $x\in s$.  Then $x\in s^*s$ and so without loss of generality we may assume that $s=U$ with $U\subseteq \mathscr G\skel 0$ compact open.  Then, by convariance, we have that $[m]_x=[\theta_{\chi_U}m]_x=[\p_{\chi_U}m]_x = \beta_x([m]_x)$, demonstrating (S1).  Suppose that $\dom(\sigma)=\ran(\tau)$ and $\sigma\in s$ and $\tau\in t$ with $s,t\in S$.  Then $\sigma\tau\in st\in S$ and so $\beta_{\sigma\tau}([m]_{\dom(\tau)}) = [\p_{st}(m)]_{\ran(\sigma)} = [\p_s(\p_t(m))]_{\ran(\sigma)} = \beta_{\sigma}([\p_t(m)]_{\dom(\sigma)}) = \beta_{\sigma}(\beta_{\tau}([m]_{\dom(\tau)}))$, establishing (S3) and so $\mathcal M$ is a $\mathscr G$-sheaf.

  Let us check that $\beta$ makes $\mathcal M$ into a $\mathscr G$-sheaf of $\mathcal O$-modules.  We need to verify that if $r\in \mathcal O_{\dom(\gamma)}$, then $\beta_{\gamma}(r[m]_{\dom(\gamma)}) = \alpha_{\gamma}(r)\beta_{\gamma}([m]_{\dom(\gamma)})$.  Choose $s\in S$ with $\gamma\in S$.  As $\dom(\gamma)\in \dom(s)$, we can find a section $t$ supported on $\dom(s)$ with $t(\dom(\gamma)) = r$.  Then $t\in D_{s^*}$ and $\til\alpha_s(t)(\ran(\gamma)) = \alpha_{\gamma}(t(\dom(\gamma))) = \alpha_{\gamma}(r)$.  Thus $\alpha_{\gamma}(r)\beta_{\gamma}([m]_{\dom(\gamma)}) = [\theta_{\til\alpha_s(t)}\p_s(m)]_{\ran(\gamma)}$.  On the other hand, using covariance,
$\beta_{\gamma}(r[m]_{\dom(\gamma)}) =\beta_{\gamma}(r[\theta_{\chi_{\dom(s)}}(m)]_{\dom(\gamma)}) = [\p_s\theta_t\p_{s^*}\p_s(m)]_{\ran(\gamma)} = [\theta_{\til\alpha_s(t)} \p_s(m)]_{\ran(\gamma)}$, as required.

It follows that $\Gamma_c(\mathscr G,\mathcal M)$ is a unitary $\Gamma_c(\mathscr G,\mathcal O)$-module.  It is therefore an $A\rtimes S$-module via $\wh \theta\rtimes \wh \p$.  Our goal is to show that $M\cong \Gamma_c(\mathscr G,\mathcal M)$ as an $A\rtimes S$-module via $\psi\colon M\to \Gamma_c(\mathscr G,\mathcal M)$ where $\psi(m) = \wh m$ is given by $\wh m(x) = [m]_x$.  It follows from Theorem~\ref{t:disint} applied to $\mathscr G\skel 0$ that $\psi$ is an isomorphism of $A$-modules.  If $a\in D_s$, then $(\wh \theta\rtimes \wh \p)(a\delta_s+\mathcal N) = a\chi_s$ and hence, since $\psi$ is an $A$-module isomorphism and $a\delta_s+\mathcal N = (a\delta_{ss^*}+\mathcal N)(\chi_{ss^*}\delta_s+\mathcal N)$, it remains to show that $\psi((\chi_{ss^*}\delta_sm+\mathcal N)m) = \chi_s\psi(m)$ or, equivalently, $\psi(\p_s(m)) = \chi_s\psi(m)$.    First note that $\p_s(m) = \p_{ss^*}\p_s(m) = \theta_{\chi_{\ran(s)}}\p_s(m)$ and hence $\psi(\p_s(m)) = \wh{\p_s(m)}$ is supported on $\ran(s)$ by the proof of Theorem~\ref{t:disint}.  Since $\chi_{\ran(s)}\ast \chi_s = \chi_s$, clearly $\chi_s\wh m=\chi_s\psi(m)$ is also supported on $\ran(s)$.  Let $x\in \ran(s)$ and let $\gamma$ be the unique arrow of $s$ with $\ran(\gamma)=x$.  Then
\[(\chi_s\wh m)(x) =\chi_s(\gamma)\beta_{\gamma}(\wh m(\dom(\gamma)))= \beta_{\gamma}([m]_{\dom(\gamma)}) = [\p_s(m)]_{\ran(\gamma)} = \wh {\p_s(m)}(\ran(\gamma))\]
 and so $\chi_s\wh m = \wh {\p_s(m)}$, as required.  We have now essentially proved the following theorem.

\begin{Thm}\label{t:factor.through}
Let $\mathcal O$ be a $\mathscr G$-sheaf of rings on an ample groupoid $\mathscr G$ and let $S\leq \mathscr G^a$ be an inverse subsemigroup satisfying the germ conditions.  Then each unitary $\Gamma_c(\mathscr G\skel 0,\mathcal O)\rtimes S$-module is an inflation of a unitary $\Gamma_c(\mathscr G,\mathcal O)$-module along the canonical surjection $\wh \theta\rtimes \wh \p\colon \Gamma_c(\mathscr G\skel 0,\mathcal O)\rtimes S\to \Gamma_c(\mathscr G,\mathcal O)$.  Consequently, $\wh \theta\rtimes \wh \p$ is an isomorphism.  In particular, \[\Gamma_c(\mathscr G,\mathcal O)\cong \Gamma_c(\mathscr G\skel 0,\mathcal O)\rtimes \mathscr G^a\] holds.
\end{Thm}
\begin{proof}
Everything except that $\wh \theta\rtimes \wh \p$ is an isomorphism has been proved.  Put $R=\Gamma_c(\mathscr G\skel 0,\mathcal O)\rtimes S$ and $\pi = \wh\theta\rtimes \wh \p$.  Then $R$ is a unitary left $R$-module via the regular action (by left multiplication).  Moreover, this is a faithful module since $R$ has local units and hence if $0\neq r\in R$, then there is an idempotent $e\in R$ with $re=r\neq 0$.  As the module action of $R$ on $R$ factors through $\pi$, we have that $\pi(r)=0$ implies $rR=0$ and hence $r=0$.  Thus $\pi$ is injective and hence an isomorphism.
\end{proof}

Applying this to the case of a constant sheaf, we obtain the following extension of results in~\cite{BG, Demeneghi}, generalized to an arbitrary base ring (not necessarily commutative, let alone a field).

\begin{Cor}
Let $\mathscr G$ be an ample groupoid and $R$ a unital ring.  Then $A_R(\mathscr G)\cong C_c(\mathscr G\skel 0, R)\rtimes \mathscr G^a$, where $A_R(\mathscr G)$ is the Steinberg algebra of $\mathscr G$ with coefficients in $R$ (i.e., $\Gamma_c(\mathscr G,\Delta(R))$) and $C_c(\mathscr G\skel 0, R)$ is the ring of compactly supported, locally constant functions $f\colon \mathscr G\skel 0\to R$ with pointwise operations.
\end{Cor}

\section{Sheaf representations of rings with local units}\label{Piercerep}
In this section, we generalize the Pierce representation of a ring~\cite{Pierce} as global sections of a sheaf of rings over a Stone space in two ways: we consider rings with local units and we allow smaller generalized Boolean algebras.  First we recall the generalized Stone space of a generalized Boolean algebra $B$.  A \emph{character} of $B$ is a non-zero homomorphism $\lambda\colon B\to \{0,1\}$ to the two-element Boolean algebra.  So a character $\lambda$ satisfies:
\begin{itemize}
\item $\lambda(0)=0$;
\item $\lambda(B)\neq 0$;
\item $\lambda(a\vee b) = \lambda(a)\vee \lambda(v)$;
\item $\lambda(ab) =\lambda(a)\lambda(b)$;
\item $\lambda(a\setminus b)=\lambda(a)\setminus \lambda(b)$
\end{itemize}
where we denoted the meet in a semilattice by product.

The (generalized) \emph{Stone space} of $B$ is the space $\wh {B}$ of characters of $B$ topologized by taking as a basis the sets
\[D(a)=\{\lambda\in \wh{B}\mid \lambda(a)=1\}\] with $a\in B$.  The sets $D(a)$ constitute the compact open subsets of $\wh{B}$ and $a\mapsto D(a)$ is an isomorphism of generalized Boolean algebras (this follows easily from Lemma~\ref{l:boolean.facts} below).  The space $\wh{B}$ is compact if and only if $B$ has a maximum, i.e., is a Boolean algebra.

Recall that a \emph{filter} in a poset is a proper non-empty subset which is upward closed and downward directed.  For a generalized Boolean algebra, $\mathcal F\neq \emptyset$ will be a filter if $0\notin \mathcal F$, it is closed under meets and it is upward closed.
Note that if $\lambda$ is a character, then $\lambda\inv(1)$ is an ultrafilter (i.e., maximal proper filter) on $B$ and the characteristic function of an ultrafilter is a character.  Since this is not so familiar for non-unital Boolean algebras, we include a proof.

The following lemma is standard.

\begin{Lemma}\label{l:boolean.facts}
Let $B$ be a generalized Boolean algebra.
\begin{enumerate}
\item Let $\mathcal F$ be a filter on $B$ and $a\notin \mathcal F$ such that $ab\neq 0$ for all $b\in \mathcal F$.  Then there is an ultrafilter containing $\mathcal F$ and $a$.
\item A filter $\mathcal F$ is an ultrafilter if, and only if, for all $a\in B\setminus \mathcal F$, there exists $b\in \mathcal F$ with $ab=0$.
\item If $a,b\in B$ and $a\nleq b$, then there is an ultrafilter $\mathcal F$ with $a\in \mathcal F$ and $b\notin \mathcal F$.
\end{enumerate}
\end{Lemma}
\begin{proof}
For the first item, let \[\mathcal F' = \{b\in B\mid \exists c\in \mathcal F\ \text{with}\ b\geq ac\}.\]  Then $0\notin \mathcal F'$ by hypothesis and clearly $\mathcal F'$ is a filter containing $\mathcal F$ and $a$.  By Zorn's lemma, there is an ultrafilter $\mathcal U$ containing $\mathcal F'$.

For the second item, suppose first that $\mathcal F$ is an ultrafilter.  If there exists $a\notin \mathcal F$ such $ab\neq 0$ for all $b\in \mathcal F$, then by the first item there is an ultrafilter containing $\mathcal F$ and $a$, contradicting that $\mathcal F$ is an ultrafilter.  Conversely, suppose that $\mathcal F$ has the desired property.  Let $\mathcal F'$ be a filter properly containing $\mathcal F$, and let $a\in \mathcal F'\setminus \mathcal F$.  Then there exists  $b\in \mathcal F$ with $ab=0$ and so $0\in \mathcal F'$, a contradiction.  Thus $\mathcal F'$ is an ultrafilter.

For the third item, note that $a\setminus b\neq 0$.  The set $\mathcal F'= \{c\in B\mid c\geq a\setminus b\}$ is a filter containing $a\setminus b$.  Hence, by the first item, it is contained in an ultrafilter $\mathcal F$.  As $a\geq a\setminus b$, we have that $a\in \mathcal F'\subseteq \mathcal F$.  Since $b(a\setminus b)=0$, we must have $b\notin \mathcal F$.
\end{proof}

\begin{Cor}
A mapping $\lambda\colon B\to \{0,1\}$ is a character if and only if $\lambda\inv(1)$ is an ultrafilter.
\end{Cor}
\begin{proof}
Suppose first that $\lambda$ is a character.  To verify that $\lambda\inv(1)$ is an ultrafilter, we must show by Lemma~\ref{l:boolean.facts} that if $\lambda(a)=0$, then there is $b$ with $\lambda(b)=1$ and $ab=0$.  Indeed, let $c\in \lambda\inv (1)$.  Put $b=c\setminus a$.  Then $\lambda(b)=\lambda(c)\setminus \lambda(a) = 1$ and so $b\in \lambda\inv(1)$.  Also $ba=0$.

Conversely, assume that $\lambda\inv(1)$ is an ultrafilter.  Since $\lambda\inv(1)$ is non-empty, we have that $\lambda\neq 0$.  As $\lambda\inv(1)$ is proper, we have that $\lambda(0)=0$.   Since $\lambda\inv(1)$ is upward closed, to show that $\lambda$ preserves joins, we must show that if $a,b\in \lambda\inv(0)$, then $a\vee b\in \lambda\inv(0)$.  Indeed, by Lemma~\ref{l:boolean.facts} we can find $a',b'\in \lambda\inv (1)$ with $aa'=0$ and $bb'=0$.  Then $a'b'\in \lambda\inv(1)$ and $(a\vee b)a'b' = aa'b'\vee ba'b' = 0$.  Thus $a\vee b\in \lambda\inv(0)$.  If $\lambda(a)=1=\lambda(b)$, then since $\lambda\inv(1)$ is a filter $\lambda(ab)=1$.  If $a$ or $b$ is not in $\lambda\inv (1)$, then since $\lambda\inv(1)$ is upward closed, it follows that $ab\notin\lambda\inv(1)$ and so $\lambda$ preserves meets.  To see that $\lambda$ preserves relative complements, if $a\notin \lambda\inv (1)$, then $a\setminus b\leq a$ implies $a\setminus b\notin \lambda\inv(1)$.  If $a,b\in \lambda\inv(1)$, then $b(a\setminus b)=0$ implies $a\setminus b\notin \lambda\inv(1)$ and so $\lambda(a\setminus b) =0=\lambda(a)\setminus \lambda(b)$.  If $\lambda(a)=1$ and $\lambda(b)=0$, then $be=0$ for some $e\in \lambda\inv(1)$ by Lemma~\ref{l:boolean.facts}.  Then $a\setminus b\geq (a\setminus b)e = ae\setminus be = ae\in \lambda\inv(1)$ and so $a\setminus b\in \lambda\inv(1)$.  Thus $\lambda(a\setminus b) = 1=\lambda(a)\setminus \lambda(b)$.  We conclude that $\lambda$ is a character.
\end{proof}

Our standing assumption now is that $R$ is a ring and $B\subseteq E(Z(R))$ is a sub-generalized Boolean algebra of the generalized Boolean algebra of central idempotents of $R$ that is also a set of local units for $R$, i.e., $R=\bigcup_{e\in B}eRe$.  Note that since $e$ is central, $eRe=Re=eR$.  For example, if $R$ is unital then we can take $B= E(Z(R))$ (and, in fact, any choice of $B$ will have to contain $0$ and $1$).  We will show that $R$ can be identified with the ring of global sections with compact support $\Gamma_c(\wh{B},\mathcal O_B)$ for a certain sheaf of rings $\mathcal O_B$ on $\wh{B}$.  When $R$ is unital and $B=E(Z(R))$, this will recover the Pierce representation~\cite{Pierce}.

If $e,f\in B$ with $f\leq e$, then there is a natural restriction $\rho^e_f\colon eR\to fR$ given by $r\mapsto fr$ satisfying the usual axioms for a directed system.  Thus, if $\lambda\in \wh{B}$, then we can define
\[R_{\lambda} = \varinjlim_{e\in \lambda\inv(1)} eR.\]
The following description of $R_{\lambda}$ is often more convenient.
Let \[I_{\lambda} = \{r\in R\mid \exists e\in \lambda\inv(1)\ \text{with}\ er=0\}.\] Observe that $I_{\lambda}$ is an ideal. Indeed, if $er=0=fs$ with $r,s\in R$ and $e,f\in \lambda\inv(1)$, then $ef\in \lambda\inv(1)$ and $ef(r-s) = fer-efs=0$ and, also,  $etr=ter=0$ and $ert=0$, for any $t\in R$.  Here we have used that $B$ consists of central idempotents. Note that, by construction, if $e\in \lambda\inv(1)$, then $r+I_{\lambda} = er+I_{\lambda}$ as $e(r-er) =0$. It is then straightforward to verify that $R/I_{\lambda}\cong R_{\lambda}$ via the map sending $r+I_{\lambda}$ to the class of $er$ where $e\in \lambda\inv(1)$.  From now on we put $[r]_{\lambda} = r+I_{\lambda}$ and identify $R_{\lambda}$ with $R/I_{\lambda}$.  Note that the ring $R_{\lambda}$ is unital with identity $[e]_{\lambda}$ with $e\in \lambda\inv(1)$.  Indeed, we already observed that $[r]_{\lambda}=[er]_{\lambda}$ and, since $e$ is central, also $[r]_{\lambda} = [re]_{\lambda}$.  We have that $[r]_{\lambda}=[s]_{\lambda}$ if and only if $er=es$ some $e\in \lambda\inv(1)$.

The underlying space of $\mathcal O_B$ is defined to be $E=\coprod_{\lambda\in \wh{B}} R_{\lambda}$. If $r\in R$ and $e\in B$, then we put
\[(r,D(e)) = \{[r]_{\lambda}\mid \lambda \in D(e)\}.\]   The reader will easily verify that the sets of the form $(r,D(e))$ form a basis for a topology on $E$ and that $p\colon E\to \wh{B}$ defined by $p([r]_{\lambda}) = \lambda$ maps $(r,D(e))$ homeomorphically to $D(e)$, whence $p$ is a local homeomorphism.  The reader will also check that the ring structures on the $R_{\lambda}$ turn $\mathcal O_B=(E,p,+,\cdot)$ into a sheaf of unital rings on $\wh{B}$.

Since $\wh{B}$ is Hausdorff, we can identify $\Gamma_c(\wh{B},\mathcal O_B)$ with the ring of continuous sections with compact support of $\mathcal O$ (under pointwise operations).  Note that since the image of the zero section is open, the support of a section is always closed.  Thus to show that a section has compact support it is enough to show that its support is contained in a compact set.
 To each $r\in R$, we define a mapping $\wh r\colon \wh {B}\to E$ by $\wh r(\lambda) = [r]_{\lambda}$.

\begin{Prop}\label{p:gelfand.transform}
Let $r\in R$.  Then $\wh r\in \Gamma_c(\wh{B},\mathcal O_B)$.
\end{Prop}
\begin{proof}
Clearly, the mapping $\wh r$ is a section of $p$.  Let us check continuity.  Let $(s,D(e))$ with $s\in R$ and $e\in B$ be a basic neighborhood in $E$ of $\wh r(\lambda)$.  Then $\lambda(e)=1$ and $[s]_{\lambda} = [r]_{\lambda}$.  Thus we can find $f\in \lambda\inv(1)$ with $fs=fr$.  Then $efs=efr$ and $ef\in \lambda\inv(1)$.  Thus  $\lambda\in D(ef)$ and if $\tau\in D(ef)$, then $\tau\in D(e)$ and $\wh r(\tau) = [r]_{\tau}=[s]_{\tau}\in (s,D(e))$.  Thus $\wh r$ is continuous.  We need to check that $\wh r$ has compact support.  Since $B$ is a set of local units for $R$, we can find $e\in B$ with $er=r$.  We claim that the support of $\wh r$ is contained in the compact set $D(e)$, whence $\wh r$ has compact support.  Suppose that $\lambda\notin D(e)$.  Then $e\notin\lambda\inv(1)$ and so, by Lemma~\ref{l:boolean.facts}, there exists $f\in \lambda\inv(1)$ with $fe=0$.  Then $fr=fer=0$ and so $[r]_{\lambda} = [0]_{\lambda}$.  This completes the proof.
\end{proof}

\begin{Rmk}\label{r:control.support}
The above proof shows that if $er=r$ with $e\in B$, then $\wh r$ has support contained in $D(e)$.
\end{Rmk}

\begin{Thm}\label{t:ring.of.sec}
Let $R$ be a ring and $B\subseteq E(Z(A))$ a sub-generalized boolean algebra that is a set of local units for $R$.  Then $R\cong \Gamma_c(\wh B,\mathcal O_B)$ via the mapping $r\mapsto \wh r$ with $\wh r(\lambda)=[r]_{\lambda}$.
\end{Thm}
\begin{proof}
Define $\Psi\colon R\to \Gamma_c(\wh B,\mathcal O_B)$ by $\Psi(r) = \wh r$.
If $r,s\in R$ and $\lambda\in \wh B$, then $(\wh r+\wh s)(\lambda) = [r]_{\lambda}+[s]_{\lambda} = [r+s]_{\lambda}=\wh {r+s}(\lambda)$ and similarly, $\wh{r}\wh{s} = \wh{rs}$.  Thus $\Psi$ is a homomorphism.  Suppose $0\neq r\in R$.  Let $\mathcal F= \{e\in B\mid er=r\}$. Since $B$ is a set of local units, $\mathcal F\neq\emptyset$ and $0\notin \mathcal F$ as $r\neq 0$.  Trivially, $\mathcal F$ is a filter.  Thus we can find a character $\lambda$ with $\mathcal F\subseteq \lambda\inv(1)$ by Lemma~\ref{l:boolean.facts}.  We claim that $0\neq \wh r(\lambda)=[r]_{\lambda}$.  Indeed, suppose that $f\in \lambda\inv(1)$ with $fr=0$.  Let $e\in \mathcal F$.  Then $(e\setminus ef)r =(e-ef)r = r$ and so $e\setminus ef\in \mathcal F\subseteq \lambda\inv(1)$.  But then $0=f(e-ef)=f(e\setminus ef)\in \lambda\inv(1)$, a contradiction.  Thus $\wh r(\lambda)\neq 0$ and so $\wh r\neq 0$.  We conclude that $\Psi$ is injective.

The most difficult part of the proof is to show that $\Psi$ is surjective.  Let $f\colon \wh B\to E$ be a continuous section with compact support $\supp(f)$.  For each $\lambda\in \supp(f)$, choose $r_{\lambda}\in R$ with $f(\lambda) = [r_{\lambda}]_{\lambda}$ and $e_{\lambda}\in \lambda\inv(1)$.  Then $f(\lambda)\in (r_{\lambda},D(e_{\lambda}))$ and so we can find a compact open set $D(b_{\lambda})$ with $b_{\lambda}\in B$ and $f(D(b_{\lambda}))\subseteq (r_{\lambda},D(e_{\lambda}))$.  Since $\supp(f)$ is compact, we can find $\lambda_1,\ldots, \lambda_n$ such that $\supp(f)\subseteq D(b_{\lambda_1})\cup \cdots\cup D(b_{\lambda_n})$.  Note that $f(\lambda) = [r_{\lambda_i}]_{\lambda}$ for all $\lambda\in D(b_{\lambda_i})$.  By putting $b_1 = b_{\lambda_1}$ and, inductively, $b_i=b_{\lambda_i}\setminus (b_1\vee\ b_2\vee \cdots\vee b_{i-1})$, we can find mutually orthogonal $b_1,\ldots, b_n$ with $b_i\leq b_{\lambda_i}$ and $b_1\vee\cdots \vee b_n = b_{\lambda_1}\vee\cdots \vee b_{\lambda_n}$.  In particular, $D(b_i)\subseteq D(b_{\lambda_i})$ for $i=1,\ldots, n$ and $\supp(f)\subseteq D(b_1)\cup\cdots\cup D(b_n)$ and this union is disjoint.  Putting $r_i = b_ir_{\lambda_i}$ we have that $f(\lambda) = [r_{\lambda_i}]_{\lambda} = [r_i]_{\lambda}$ for all $\lambda\in D(b_i)$.  Put $r=r_1+\cdots +r_n$.  We claim that $\wh r= f$.  Let $b=\bigvee_{i=1}^n b_i=b_1+\cdots+b_n$.  Then \[br= (b_1+\cdots+b_n)r =(b_1+\cdots +b_n)(r_1+\cdots +r_n) = r_1+\cdots +r_n = r\] by orthogonality of $b_1,\ldots, b_n$. Remark~\ref{r:control.support} shows that $\supp(\wh r)\subseteq D(b)$.   On the other hand, $\supp(f)\subseteq D(b_1)\cup\cdots\cup D(b_n) = D(b)$, as well.   Suppose that $\lambda\in D(b)$.  Then there is a unique $i$ with $\lambda\in D(b_i)$.  Then we compute that \[\wh r(\lambda) = [r_1+\cdots +r_n]_{\lambda} = [b_i(r_1+\cdots+r_n)]_{\lambda} = [r_i]_{\lambda}=f(\lambda).\]  This completes the proof that $\Psi(r) = f$.
\end{proof}

The most important special case of Theorem~\ref{t:ring.of.sec} is when $B=E(Z(A))$, in which case $\wh B$ is known as the \emph{Pierce spectrum} of $A$~\cite{Pierce}.  Since we shall need it a lot in the next section, we denote by $\wh A$ the Pierce spectrum of $A$ and $\mathcal O_A$ the sheaf of unital rings constructed above for $B=E(Z(A))$ (assuming that $A$ has a set of central local units).  If $A$ is a commutative ring, then each stalk $\mathcal O_{A,\lambda}$ is an \emph{indecomposable ring}, that is, has no (central) idempotents except $0$ and $1$.  Indeed, if $[a]_{\lambda}$ is idempotent, then there exists $e\in \lambda\inv(1)$ with $ea=ea^2=(ea)^2$ and so $[a]_{\lambda}=[f]_{\lambda}$ with $f=ea$ an idempotent of $A$.  If $f\in \lambda\inv(1)$, then $[f]_{\lambda}$ is the identity of $\mathcal O_{A,x}$, whereas if $f\notin \lambda\inv(1)$, then $fe'=0$ for some $e'\in \lambda\inv(1)$ by Lemma~\ref{l:boolean.facts} and so $[f]_{\lambda} = [e'f]_{\lambda}=[0]_{\lambda}$. Thus we have the following corollary, generalizing Pierce's sheaf representation theorem from the unital case to the case of rings with local units.

\begin{Cor}\label{c:pierce}
Let $A$ be a ring with central local units.  Then $A\cong \Gamma_c(\wh A,\mathcal O_A)$ where $\wh A$ is the Pierce spectrum of $A$.  If $A$ is commutative, then $\mathcal O_A$ is a sheaf of indecomposable unital rings and hence every commutative ring with local units is the ring of global sections with compact support of a sheaf of indecomposable unital rings on a generalized Stone space.
\end{Cor}

It follows from Theorem~\ref{t:disint} that we can identify the category of unitary $A$-modules with the category of sheaves of $\mathcal O_A$-modules in the context of Corollary~\ref{c:pierce}.

Note that if $X$ is a Hausdorff space with a basis of compact open sets and $K$ is an indecomposable commutative ring, then it is easy to verify that $X$ is homeomorphic to the Pierce spectrum of the ring $C_c(X,K)$ of locally constant $K$-valued functions with compact support.

\section{Skew inverse semigroup rings as groupoid convolution algebras}\label{isrAsGpd}

We now aim to prove the converse of Theorem~\ref{t:factor.through} by showing that every skew inverse semigroup ring (with respect to a spectral action) is isomorphic to a groupoid convolution algebra.
Let $\alpha$ be a spectral action of an inverse semigroup $S$ on a ring $A$.  We want to define a Boolean action of $S$ on the Pierce spectrum $\wh A$  of $A$.  For $s\in S$, let us put \[\wh D_{s} = \{\lambda\in \wh A\mid \lambda(1_{ss^*})=1\}.\]   The reader should verify that $\wh D_s$ is compact open (it is the basic compact open associated to $1_{ss^*}\in E(Z(A))$) and can be identified with the Pierce spectrum of $D_s$.   Note that $\wh D_s=\emptyset$ if and only if $1_{ss^*}=0$, in which case $1_{s^*s}=\alpha_{s^*}(1_{ss^*})=0$ and so $\wh D_{s^*}=\emptyset$, as well. Define $\wh \alpha_s\colon \wh D_{s^*}\to \wh D_s$ by \[\wh \alpha_s(\lambda)(e) = \lambda (\alpha_{s^*}(e1_{ss^*}))\] for $e\in Z(E(A))$.

\begin{Prop}\label{p:gen.bool.hom}
For each $s\in S$ and $\lambda\in \wh D_{s^*}$, we have that $\wh\alpha_s(\lambda)\in \wh D_s$.
\end{Prop}
\begin{proof}
 Note that if $\lambda\in \wh D_{s^*}$, then $\lambda(1_{s^*s})=1$ and so \[\wh \alpha_s(\lambda)(1_{ss^*}) = \lambda(\alpha_{s^*}(1_{ss^*})) = \lambda(1_{s^*s})=1.\]  We check that $\wh \alpha_s(\lambda)$ is a generalized Boolean algebra homomorphism.
It is clearly non-zero as $\wh \alpha_s(\lambda)(1_{ss^*}) =1$.  Since $1_{ss^*}$ is a central idempotent, the mapping $A\to D_{s^*}$ given by $a\mapsto \alpha_{s^*}(a1_{ss^*})$ is a surjective ring homomorphism and hence induces a generalized Boolean algebra homomorphism on central idempotents.  Composing with $\lambda$ shows that $\wh\alpha_s(\lambda)$ is a homomorphism of generalized Boolean algebras, and so $\wh\alpha_s(\lambda)\in \wh D_s$, as required.
\end{proof}

\begin{Prop}\label{p:is.a.hom}
$\wh \alpha_s\colon \wh D_{s^*}\to \wh D_s$ is a homeomorphism.
\end{Prop}
\begin{proof}
This is clear if $\wh D_{s^*}=\emptyset=\wh D_s$.  So assume this is not the case.
To see that $\wh \alpha_s$ is bijection, we note that for $\lambda\in D_{s^*}$
\begin{align*}
\wh \alpha_{s^*}(\wh\alpha_s(\lambda))(e) & = \wh\alpha_s(\lambda)(\alpha_s(e1_{s^*s})) = \lambda(\alpha_{s^*}(\alpha_s(e1_{s^*s})1_{ss^*}))\\ & = \lambda(\alpha_{s^*s}(e1_{s^*s})) =\lambda(e1_{s^*s})= \lambda(e)\lambda(1_{s^*s}) = \lambda(e).
\end{align*}
 A dual verification shows that $\wh \alpha_s$ and $\wh \alpha_{s^*}$ are inverses.

 We now check continuity of $\wh \alpha_s$.  Continuity of $\wh \alpha_{s^*}$ will then imply that $\wh \alpha_s$ is a homeomorphism.   A basic neighborhood $V$ of $\wh \alpha_s(\lambda)$, with $\lambda\in \wh D_{s^*}$, is  of the form $V=\{\mu\in \wh A\mid \mu(f)=1\}$ where $f\in E(Z(A))$ with $\wh \alpha_s(\lambda)(f)=1$.  As $\wh \alpha_s(\lambda)\in \wh D_s$, we may assume that $f1_{ss^*} =f$, i.e., $f\in D_s$.  Let $f' = \alpha_{s^*}(f)$; it is a central idempotent of $D_{s^*}$ and hence of $A$.  The set $U = \{\mu \in \wh D_{s^*}\mid \mu(f')=1\}$ is an open set.  Also $\lambda(f') = \lambda(\alpha_{s^*}(f)) = \wh \alpha_s(\lambda)(f)=1$ and so $\lambda\in U$. If $\mu\in U$, then $\wh\alpha_s(\mu)(f) = \mu(\alpha_{s^*}(f))=\mu(f')=1$ and so $\wh \alpha_s(U)\subseteq V$.  This establishes the continuity of $\wh\alpha_s$.
\end{proof}

We now need to verify that $s\mapsto \wh \alpha_s$ is a homomorphism of inverse semigroups $S\to I_{\wh A}$.

\begin{Prop}\label{p:is.hom}
The mapping $\wh \alpha\colon S\to I_{\wh A}$ given by $\wh\alpha(s)=\wh\alpha_s$ is a homomorphism.
\end{Prop}
\begin{proof}
As usual, it suffices to prove that $\wh \alpha$ is order preserving, $\wh \alpha(ef) =\wh\alpha(e)\wh\alpha(f)$ for idempotents $e,f\in E(S)$,  $\wh\alpha(st) = \wh\alpha(s)\wh\alpha(t)$ whenever $s^*s=tt^*$.    If $s\leq t$, then $1_{ss^*}\leq 1_{tt^*}$ and $1_{s^*s}\leq 1_{t^*t}$ and so $\wh D_s\subseteq \wh D_t$ and $\wh D_{s^*}\subseteq \wh D_{t^*}$.  Moreover, if $\lambda\in \wh D_{s^*}$, then $\wh \alpha_s(\lambda)(e) = \lambda(\alpha_{s^*}(e1_{ss^*}))=\lambda(\alpha_{s^*}(e1_{ss^*}1_{tt^*})) =\lambda(\alpha_{t^*}(e1_{ss^*}1_{tt^*}))=\lambda(1_{t^*ss^*t})\lambda(\alpha_{t^*}(e1_{tt^*})) = \lambda(1_{s^*s})\wh \alpha_t(e) = \wh\alpha_t(e)$ and so $\wh \alpha(s)\leq \wh\alpha(t)$.

If $e\in E(S)$, then for $\lambda\in \wh D_e$, we have that $\lambda(1_e)=1$ and so, for $f\in E(Z(A))$, it follows that $\wh\alpha_e(\lambda)(f) = \lambda(\alpha_e(f1_e)) = \lambda(f1_e)=\lambda(f)\lambda(1_e) = \lambda(f)$.  Thus $\wh \alpha_e(\lambda) =\lambda$, i.e., $\wh \alpha_e$ is the identity on $\wh D_e$.  We conclude that if $e,f\in E(S)$, then $\wh \alpha_e\wh \alpha_f$ is the identity on $\wh D_e\cap \wh D_f$.  But $\lambda\in \wh D_e\cap \wh D_f$ if and only if $\lambda(1_e)=1=\lambda(1_f)$, which occurs if and only if $\lambda (1_e1_f)=1$, that is, $\lambda(1_{ef})=1$. So $\wh D_e\cap \wh D_f=\wh D_{ef}$, whence $\wh \alpha_e\wh\alpha_f=\wh \alpha_{ef}$.  Thus $\wh\alpha|_{E(S)}$ is a homomorphism.

If $s^*s=tt^*$, then $\wh D_t=\wh D_{s^*}$, $\wh D_{st} = \wh D_s$ and $\wh D_{(st)^*} = \wh D_{t^*}$.  Thus $\wh \alpha_s\wh\alpha_t,\wh \alpha_{st}\colon D_{t^*}\to D_s$. If $\lambda\in D_{t^*}$, then we compute, using  $1_{tt^*} = 1_{s^*s}$, that
\begin{align*}
\wh \alpha_s(\wh\alpha_t(\lambda))(e) & = \wh\alpha_t(\lambda)(\alpha_{s^*}(e1_{ss^*})) = \lambda(\alpha_{t^*}(\alpha_{s^*}(e1_{ss^*})1_{tt^*})) \\ & = \lambda(\alpha_{t^*}(\alpha_{s^*}(e1_{ss^*}))).
\end{align*}
On the other hand,
\[\wh\alpha_{st}(\lambda)(e) = \lambda(\alpha_{(st)^*}(e1_{stt^*s^*}))=\lambda(\alpha_{t^*}(\alpha_{s^*}(e1_{ss^*})))\] and so $\wh \alpha_{st}(\lambda) = \wh\alpha_s(\wh\alpha_t(\lambda))$, as required.
\end{proof}

Let $\mathscr G=S\ltimes \wh A$ be the corresponding groupoid of germs.  It is ample as $\wh A$ has a basis of compact open sets.  Moreover, since the action is Boolean, the compact open subsets of $\mathscr G\skel 1$ of the form $U(s)=(s,\wh D_{s^*})$ form an inverse semigroup $\til S$ of compact open bisections which is a quotient of $S$ (via $s\mapsto U(s)$)  satisfying the germ conditions.  Our goal is to extend the sheaf of rings structure on $\mathcal O_A$ to a $\mathscr G$-sheaf of rings structure so that $A\rtimes S\cong \Gamma_c(\mathscr G,\mathcal O_A)$.

Let us recall that the stalk $\mathcal O_{A,\lambda} = A/I_{\lambda}$ where \[I_{\lambda} = \{a\in A\mid \exists e\in \lambda\inv(1)\ \text{with}\ ea=0\}\] and that the class $a+I_{\lambda}$ is denoted $[a]_{\lambda}$.  The unit is the class $[e]_{\lambda}$ where $\lambda(e)=1$.  We now define a $\mathscr G$-sheaf structure on $\mathcal O_A$ by putting
\begin{equation}\label{eq:sheaf.action.gpd.germs}
\alpha_{[s,\lambda]}([a]_{\lambda}) = [\alpha_s(1_{s^*s}a)]_{\wh \alpha_s(\lambda)}
\end{equation}
for $[s,\lambda]\in\mathscr G\skel 1$.  (We hope the reader will forgive our abuse of the notation $\alpha$.) Here, of course, we must have $\lambda \in \wh D_{s^*}$ and we need to show independence of the choice of $a$ and $s$.

  Indeed, first suppose that $[a]_{\lambda} = [b]_{\lambda}$.  Then there is $e\in \lambda\inv(1)$ with $ea=eb$.   Then, as $\lambda(1_{s^*s})=1$, we have that \[\wh\alpha_s(\lambda)(\alpha_s(e1_{s^*s})) = \lambda(\alpha_{s^*}(\alpha_s(e1_{s^*s})1_{ss^*})) = \lambda(\alpha_{s^*s}(e1_{s^*s})) = \lambda(e1_{s^*s})=1\] and so
\begin{align*}
[\alpha_s(1_{s^*s}a)]_{\wh \alpha_s(\lambda)}  & = [\alpha_s(e1_{s^*s})\alpha_s(1_{s^*s}a)]_{\wh \alpha_s(\lambda)}=[\alpha_s(ea1_{s^*s})]_{\wh \alpha_s(\lambda)}\\ & =[\alpha_s(eb1_{s^*s})]_{\wh \alpha_s(\lambda)} =[\alpha_s(e1_{s^*s})\alpha_s(1_{s^*s}b)]_{\wh \alpha_s(\lambda)}\\ & =[\alpha_s(1_{s^*s}b)]_{\wh \alpha_s(\lambda)}
\end{align*}
establishing independence of the choice of $a$.  If $[s,\lambda]=[t,\lambda]$, then we can find $u\leq s,t$ with $\lambda \in \wh D_{u^*}$.  By symmetry, it is enough to show that $[\alpha_s(1_{s^*s}a)]_{\wh \alpha_s(\lambda)} = [\alpha_u(1_{u^*u}a)]_{\wh \alpha_u(\lambda)}$.  Note that $\wh\alpha_u(\lambda) =\wh\alpha_s(\lambda)$
and so we just need to prove that $[\alpha_s(1_{s^*s}a)]_{\wh \alpha_s(\lambda)}=[\alpha_u(1_{u^*u}a)]_{\wh \alpha_s(\lambda)}$.  Since $\lambda(1_{u^*u})=1$, we have that $[a]_{\lambda} = [1_{u^*u}a]_{\lambda}$.  Hence, by what we just proved, and the equality $u=su^*u$, we conclude that
$[\alpha_s(1_{s^*s}a)]_{\wh \alpha_s(\lambda)}=[\alpha_s(1_{s^*s}1_{u^*u}a)]_{\wh \alpha_s(\lambda)}=[\alpha_s(1_{u^*u}a)]_{\wh \alpha_s(\lambda)}=[\alpha_s\alpha_{u^*u}(1_{u^*u}a)]_{\wh \alpha_s(\lambda)}=[\alpha_u(1_{u^*u}a)]_{\wh \alpha_s(\lambda)}$, as required.

Note that $\alpha_{[s,\lambda]}\colon \mathcal O_{A,\lambda}\to \mathcal O_{A,\wh \alpha_s(\lambda)}$ is clearly a homomorphism of unital rings since $a\mapsto \alpha_s(1_{s^*s}a)$ a ring homomorphism $A\to D_s$ sending $1_{s^*s}\in \lambda\inv(1)$ to $1_{ss^*}\in \wh \alpha_s(\lambda)\inv(1)$.

We check  that $\alpha\colon \mathscr G\skel 1\times_{\dom,p} E\to E$ is continuous.  A basic neighborhood of $\alpha_{[s,\lambda]}([a]_{\lambda})=[\alpha_s(1_{s^*s}a)]_{\wh \alpha_s(\lambda)}$ is of the form  $(\alpha_s(1_{s^*s}a),W)$ where $W$ is a compact open neighborhood of $\wh \alpha_s(\lambda)$ contained in $\wh D_s$.  Then $V=\wh\alpha_{s^*}(W)$ is a compact open neighborhood of $\lambda$ contained in $\wh D_{s^*}$ with $\wh \alpha_s(V)\subseteq W$.  Moreover, $(s,V)$ is a compact open bisection containing $\gamma=[s,\lambda]$ and $(a,V)$ is a compact open subset of $E$ containing $[a]_{\lambda}$.  We claim that $U= ((s,V)\times (a,V))\cap (\mathscr G\skel 1\times_{\dom,p} E)$ is a neighborhood of $(\gamma,[a]_{\lambda})$ with $\alpha(U)\subseteq (\alpha_s(1_{s^*s}a),W)$.  Indeed, a typical element of $U$ is of the form $([s,\tau],[a]_{\tau})$ with $\tau\in V$, whence $\wh\alpha_s(\tau)\in W$.  Then $\alpha_{[s,\tau]}([a]_{\tau}) = [\alpha_s(1_{s^*s}a)]_{\wh \alpha_s(\tau)}\in (\alpha_s(1_{s^*s}a),W)$.

To finish we verify (S3) and leave (S1) and (S2) to the reader. Let $\beta=[s,\lambda'']$ and $\gamma = [t, \lambda']$. We have to check that $\alpha_{\beta}(\alpha_{\gamma}([a]_\lambda)) = \alpha_{\beta\gamma}([a]_\lambda)$ whenever $\dom(\beta)=\ran(\gamma)$ and $\dom(\gamma)=p([a]_\lambda)=\lambda$. From $\dom(\beta)=\ran(\gamma)$ we get that $\lambda'' = \wh \alpha_t(\lambda')$ and from $\dom(\gamma)=p([a]_\lambda)$ we get that $\lambda = \lambda'$. Now, notice that $\alpha_{\beta}(\alpha_{\gamma}([a]_\lambda)) = \alpha_{[s, \wh \alpha_t(\lambda')]}\left([\alpha_t(1_{t^*t} a)]_{\wh \alpha_t(\lambda')}\right) = [ \alpha_s\left(1_{s^*s} \alpha_t(1_{t^*t}a)\right)]_{\wh \alpha_s (\wh \alpha_t(\lambda'))} $. On the other hand, since $\beta \gamma = [st, \lambda']$ we have that \[\alpha_{\beta\gamma}([a]_\lambda)= [\alpha_{st}(1_{t^* s^* s t} a)]_{\wh \alpha_{st}(\lambda')}.\]  (S3) now follows from the calculation below:
\begin{equation*}
\begin{split} \alpha_{st}(1_{t^*s^*st}a)& = \alpha_s(\alpha_t(1_{t^*s^*st}a)) =  \alpha_s(\alpha_t(1_{t^*s^*st}1_{t^*t}a)) \\ & = \alpha_s(\alpha_t(1_{t^*s^*st})\alpha_t(1_{t^*t}a)) = \alpha_s(1_{s^*stt^*}\alpha_t(1_{t^*t}a)) \\ & =\alpha_s(1_{s^*s}1_{tt^*}\alpha_t(1_{t^*t}a)) = \alpha_s(1_{s^*s}\alpha_t(1_{t^*t}a)),
\end{split}
\end{equation*}
where we used that $1_{tt^*}$ is the identity on the range of $\alpha_t$, and if $e\leq t^*t$ then $\alpha_t(1_e) = 1_{tet^*}$.

\begin{Prop}\label{p:well.def.back}
Let $s,t\in S$ and suppose that $U(s)=U(t)$.  Then $1_{ss^*}\delta_s+\mathcal N=1_{tt^*}\delta_t+\mathcal N$ in $A\rtimes S$.
\end{Prop}
\begin{proof}
First note that $U(s)=U(t)$ implies $\wh D_s=\wh D_t$ and $\wh D_{s^*}=\wh D_{t^*}$, and hence $1_{ss^*}=1_{tt^*}$; let's call this latter idempotent $e$.  Also note that $U(s^*)=U(t^*)$.  Let $\lambda \in \wh D_s=\wh D_t$.  Then since $U(s^*)=U(t^*)$, we have that $[s^*,\lambda]=[t^*,\lambda]$.  Thus we can find $w_{\lambda}\in S$ with $\lambda\in \wh D_{w_{\lambda}}$ and $w_{\lambda}^*\leq s^*,t^*$.  By compactness of $\wh D_s$, we can find $\lambda_1,\ldots, \lambda_n$ such that $\wh D_s=\wh D_{w_{\lambda_1}}\cup\cdots\cup \wh D_{w_{\lambda_n}}$.  Let $U_1=\wh D_{w_{\lambda_1}}$ and $U_{i+1} = \wh D_{w_{\lambda_{i+1}}}\setminus (U_1\cup\cdots\cup U_i)$ for $1\leq i<n$.  Then $\wh D_s=U_1\cup\cdots\cup U_n$ with the $U_i$ compact open and pairwise disjoint and $U_i\subseteq \wh D_{w_{\lambda_i}}$.  Then there are pairwise orthogonal idempotents $e_1,\ldots, e_n\in E(Z(Ae))$ with $e=e_1+\cdots + e_n$ and $U_i = \{\lambda\in \wh A\mid \lambda(e_i)=1\}$.  Note that $e_i\in A1_{w_{\lambda_i}w^*_{\lambda_i}}=D_{w_{\lambda_i}}$ from $U_i\subseteq \wh D_{w_{\lambda_i}}$.  Thus $e_i\delta_s+\mathcal N = e_i\delta_{w_{\lambda_i}}+\mathcal N = e_i\delta_t+\mathcal N$, for $i=1,\ldots, n$,  as $w_{\lambda_i}\leq s,t$.  Then, recalling that $1_{ss^*}=e=1_{tt^*}$,  we have that $1_{ss^*}\delta_s+\mathcal N= e_1\delta_s+\cdots +e_n\delta_s+\mathcal N =e_1\delta_t+\cdots +e_n\delta_t+\mathcal N = 1_{tt^*}\delta_t+\mathcal N$ as required.
\end{proof}

\begin{Thm}\label{t:main}
If $S$ is an inverse semigroup with a spectral action $\alpha$ on a ring $A$, then $A\rtimes S\cong \Gamma_c(S\ltimes \wh A,\mathcal O_A)$ with the $S\ltimes \wh A$-sheaf  structure on $\mathcal O_A$ coming from \eqref{eq:sheaf.action.gpd.germs} and the usual sheaf of rings structure on $\mathcal O_A$ over the Pierce spectrum $\wh A$.
\end{Thm}
\begin{proof}
We already have an isomorphism $A\cong \Gamma_c(\wh A,\mathcal O_A)$ given by $a\mapsto \wh a$ by Theorem~\ref{t:ring.of.sec}.  Theorem~\ref{t:factor.through} then gives an isomorphism $\Gamma_c(S\times \wh A,\mathcal O_A)\cong A\rtimes \til S$ where $\til S$ acts on $\Gamma_c(\wh A,\mathcal O_A)\cong A$ via the action $\til \alpha$ from Section~\ref{s:as.quotient}.  Put $R=A\rtimes S$ and $R'=\Gamma_c(\wh A,\mathcal O_A)\rtimes \til S$. We use $\mathcal N$ for the ideal used in the definition of $A\rtimes S$ and $\mathcal N'$ for the ideal in the definition of $\Gamma_c(\wh A,\mathcal O_A)\rtimes \til S$.   We want to define inverse covariant systems.

Define $\theta\colon A\to R'$ by the composition of $a\mapsto \wh a$ with the embedding  $\Gamma_c(\wh A,\mathcal O_A)\hookrightarrow R'$.  Define $\p\colon S\to R'$ by $\p(s) = \chi_{U(ss^*)}\delta_{U(s)}+\mathcal N'$.  Clearly $\theta$ is a homomorphism.  Also $\p$ is  a homomorphism, as it is the composition of $s\mapsto U(s)$ with the canonical homomorphism of $\til S$ into $R'$.  We check the covariance conditions.  If $e\in E(S)$, then $U(e)=\wh D_e$ and so $\p(e) = \chi_{\wh D_e}\delta_{U(e)}+\mathcal N'$.  On the other hand, $\wh {1_e}(\lambda) = [1_e]_{\lambda}$.  If $\lambda(1_e)=1$, then $[1_e]_{\lambda}=1_{\lambda}$.  If $\lambda(1_e)=0$, then $f1_e=0$ for some $f\in \lambda\inv(1)$ by Lemma~\ref{l:boolean.facts} and so $[1_e]_{\lambda}=[f1_e]_{\lambda}=[0]_{\lambda}$.  Thus $\wh {1_e} = \chi_{\wh D_e}$ and so $\theta(1_e) = \chi_{\wh D_e}\delta_{U(e)}+\mathcal N'=\p(e)$.

Now suppose that $a\in D_{s^*}$.  Note that $\dom(U(s)) = \wh D_{s^*}$.  Since $1_{s^*s}a=a$, it follows that $\wh a$ is supported on $\wh D_{s^*s} = \wh D_{s^*}$ (see the proof of Theorem~\ref{t:ring.of.sec}).  Now we check that $\theta(\alpha_s(a)) = \p(s)\theta(a)\p(s^*)$. Note that \[\theta(\alpha_s(a))= \wh{\alpha_s(a)}\delta_{U(ss^*)}+\mathcal N'\] as $\alpha_s(a)\in D_s=A1_{ss^*}$ implies that $\wh{\alpha_s(a)}$ is supported on $\wh D_s=U(ss^*)$.   On the other hand,
\begin{align*}
\p(s)\theta(a)\p(s^*) & = (\chi_{\wh D_s}\delta_{U(s)})(\wh a\delta_{U(s^*s)})(\chi_{\wh D_{s^*}}\delta_{U(s^*)})+\mathcal N'\\
&= (\chi_{\wh D_s}\delta_{U(s)})(\wh a\delta_{U(s^*)})+\mathcal N' \\ & = \til\alpha_{U(s)}(\chi_{\wh D_{s^*}}\wh a)\delta_{U(ss^*)}+\mathcal N'=\til\alpha_{U(s)}(\wh a)\delta_{U(ss^*)}+\mathcal N'
\end{align*}
and so we must show $\wh{\alpha_s(a)} = \til\alpha_{U(s)}(\wh a)$.

Both $\wh{\alpha_s(a)}$ and $\til \alpha_{U(s)}(\wh a)$ are supported on $\wh D_s$.   If $\lambda\in \wh D_s=\ran(U(s))$, then $\lambda = \wh\alpha_s(\nu)$ for a unique $\nu\in \wh D_{s^*}$ with $[s,\nu]\in U(s)$. Then we have \[\til\alpha_{U(s)}(\wh a)(\lambda) = \alpha_{[s,\nu]}(\wh a(\nu)) = \alpha_{[s,\nu]}([a]_{\nu}) = [\alpha_s(1_{s^*s}a)]_{\lambda}=[\alpha_s(a)]_{\lambda}=\wh{\alpha_s(a)}(\lambda).\]  We conclude that
\begin{equation}\label{eq:hat.til}
\wh{\alpha_s(a)} = \til\alpha_{U(s)}(\wh a)
\end{equation}
 as required.

It follows that we have a homomorphism $\pi = \theta\rtimes \p\colon R\to R'$ given by $\pi(a\delta_s+\mathcal N) = \theta(a)\p(s)$.

Now we define $\theta'\colon \Gamma_c(\wh A,\mathcal O_A)\to R$ and $\p'\colon \til S\to R$ by setting $\theta'$ to be the composition of the mapping $\wh a\mapsto a$ with the inclusion of $A$ and putting $\p'(U(s)) = 1_{ss^*}\delta_s+\mathcal N$.  Note that $\p'$ is well defined by Proposition~\ref{p:well.def.back}.  Clearly $\theta'$ is a homomorphism.  Since $U(s)U(t)=U(st)$, it also follows that $\p'$ is a homomorphism.  We now check the covariance conditions.  The idempotents of $\til S$ are of the form $U(e)$ with $e\in E(S)$. The identity of the domain of $\til \alpha_{U(e)}$ is $\chi_{U(e)} = \wh {1_e}$.  Thus $\theta'(\chi_{U(e)}) = 1_e\delta_e+\mathcal N$.  On the other hand, $\p'(U(e)) = 1_e\delta_e+\mathcal N$ by definition.  Next suppose that $\wh a$, with $a\in A$, is supported on $\wh D_{s^*}=\dom(U(s))$, that is, $\wh a\in \Gamma_c(\wh A,\mathcal O_a)\chi_{\wh D_{s^*}} =\Gamma_c(\wh A,\mathcal O_a)\wh{1_{s^*s}}$.  Then $\wh a =\wh a\wh {1_{s^*s}} = \wh{a1_{s^*s}}$ and so $a=a1_{s^*s}$, i.e., $a\in D_{s^*}$.

We compute that
\[
\p'(s)\theta'(\wh a)\p'(s^*) = (1_{ss^*}\delta_s)(a\delta_{s^*s})(1_{s^*s}\delta_{s^*})+\mathcal N\\
= \alpha_s(a)\delta_{ss^*}+\mathcal N . \]
On the other hand, by~\eqref{eq:hat.til}, we have that $\theta'(\til\alpha_{U(s)}(\wh a)) = \theta'(\wh{\alpha_s(a)}) = \alpha_s(a)\delta_{ss^*}+\mathcal N$.  Thus $(R,\theta',\p')$ is covariant and induces a homomorphism $\pi'=\theta'\rtimes \p'\colon R'\to R$ given by $\pi'(\wh a\delta_{U(s)}+\mathcal N') = \theta'(\wh a)\p'(U(s))$.
Let us verify that $\pi$ and $\pi'$ are inverse homomorphisms.  Let $a\in D_s$.  Then $1_{ss^*}a=a$ implies that $\wh a$ is supported on $\wh D_s=U(ss^*)$.  First note that if $a\in D_s$, then \[\pi(a\delta_s+\mathcal N) = \theta(a)\p(s) = (\wh a\delta_{U(ss^*)})(\chi_{U(ss^*)}\delta_{U(s)})+\mathcal N' = \wh a\delta_{U(s)}+\mathcal N'.\]  If $\wh a$ is supported on $\ran(U(s))=\wh D_s$, then $\wh a=\wh a\wh {1_{ss^*}} = \wh {a1_{ss^*}}$ and so $a=a1_{ss^*}$.  Thus $a\in D_s$.  Therefore,
\[\pi'(\wh a\delta_{U(s)}+\mathcal N')= \theta'(\wh a)\p'(U(s)) = (a\delta_{ss^*})(1_{ss^*}\delta_s)+\mathcal N = a\delta_s+\mathcal N.\]  It follows immediately that $\pi,\pi'$ are inverse homomorphisms.  This completes the proof.
\end{proof}

Putting together Theorems~\ref{t:factor.through} and~\ref{t:main} we see that skew inverse semigroup rings with respect to spectral actions and convolution algebras of sheaves of rings over ample groupoids are essentially one and the same. Future work will show that groupoid convolution algebras
play a role in understanding skew inverse semigroup rings.

\begin{Rmk}
We note that the ample groupoid $S\ltimes \wh A$ was not the unique possible choice.  If we took any generalized Boolean algebra $B$ in $E(Z(A))$ containing the set $\{1_e\mid e\in E(S)\}$, then we could have used Theorem~\ref{t:ring.of.sec} to represent $A$ as the global sections with compact support of a sheaf of rings on $\wh B$, defined an analogous action of $S$ on $\wh B$ and proceeded with an identical proof to extend the sheaf of rings structure over $\wh B$ to one over $S\ltimes \wh B$ to realize the skew inverse semigroup ring  as a convolution algebra over $S\ltimes \wh B$.  Each such generalized Boolean algebra gives a quotient space of $\wh A$, and hence a smaller unit space for our groupoid, but it has the property that the stalks of the sheaf of rings become in a sense bigger.  Still, it might be reasonable to work with the generalized Boolean algebra generated by $\{1_e\mid e\in E(S)\}$, which gives the smallest unit space and the largest stalks.  Then it would have the advantage that the domains of the elements of $S$ generate the generalized boolean algebra of compact open subsets of the unit space.
\end{Rmk}

\end{document}